\documentclass{amsart}

\usepackage{amsmath, amsthm, amssymb, amsfonts}
\usepackage{mathrsfs}

\newtheorem{thm}{Theorem}
\newtheorem{cor}[thm]{Corollary}
\newtheorem{lem}[thm]{Lemma}

\newtheorem*{thm*}{Theorem}
\newtheorem*{cor*}{Corollary}
\newtheorem*{lem*}{Lemma}
\newtheorem*{prop*}{Proposition}
 
\newtheorem*{ex*}{Exercise} 

\theoremstyle{definition}
\newtheorem{defn}[thm]{Definition}
\newtheorem*{defn*}{Definition}

\newtheorem*{prob*}{Problem}

\theoremstyle{definition}

\theoremstyle{definition}

\newtheorem{examp}[thm]{Example}
\newtheorem*{example*}{Example}

\theoremstyle{remark}

\newtheorem*{conj*}{Conjecture}

\newcommand{\mc}[1]{\mathcal{#1}}

\newcommand{\mbf}[1]{\mathbf{#1}}

\newcommand{\R}{{\mathbb R}}

\newcommand{\inv}{\ensuremath{^{-1}}}

\newcommand{\empset}{\{\}}
\renewcommand{\today}{\number \day \space \ifcase \month \or January\or%
  February\or March\or April\or May\or June\or July\or August\or%
  September\or October\or November\or December\fi \space \number \year}

\DeclareMathOperator{\height}{height}

\newcommand{\ileft}[2][{}]{\ensuremath{\ell_{#1}(#2)}}
\newcommand{\iright}[2][{}]{\ensuremath{\mathit{r}_{#1}(#2)}}
\newcommand{\up}[1]{\ensuremath{U(#1)}}
\newcommand{\down}[1]{\ensuremath{D(#1)}}

\newcommand{\set}[1]{\ensuremath{\left\{ #1 \right\}}}

\newcommand{\twoplustwo}{\ensuremath{\mbf{2}+\mbf{2}}}
\newcommand{\oneplusthree}{\ensuremath{\mbf{1}+\mbf{3}}}

\numberwithin{thm}{section}

\usepackage{tikz,mathptmx,marvosym,makecell}
\usetikzlibrary{arrows}
\tikzset{vtx/.style={circle, inner sep=0pt, minimum size=0.22cm,fill=black}}

\DeclareMathAlphabet{\mathcal}{OMS}{cmsy}{m}{n}
\DeclareMathOperator{\asc}{asc}
\newcommand{\lQ}[2][]{\ensuremath{\ell^{#1}\!\paren{#2}}}
\newcommand{\paren}[1]{\ensuremath{\left( #1 \right)}}
\renewcommand{\set}[1]{\ensuremath{\left\{ #1 \right\}}}
\newcommand{\moveone}{\textup{\texttt{Move 1}}}
\newcommand{\movetwo}{\textup{\texttt{Move 2}}}
\newcommand{\movethree}{\textup{\texttt{Move 3}}}

\newcommand{\nonempty}[1]{\left\langle #1 \right\rangle}
\newcommand{\Wb}{\ensuremath{\mathcal{W}}}
\newcommand{\Ub}{\ensuremath{\mathcal{U}}}
\newcommand{\Cb}{\ensuremath{\mathcal{C}}}
\newcommand{\Bb}{\ensuremath{\mathcal{B}}}
\newcommand{\Ab}{\ensuremath{\mathcal{A}}}
\newcommand{\Xoo}{\ensuremath{\mathcal{X}_{oo}}}
\newcommand{\Xos}{\ensuremath{\mathcal{X}_{os}}}
\newcommand{\Xso}{\ensuremath{\mathcal{X}_{so}}}
\newcommand{\Xss}{\ensuremath{\mathcal{X}_{ss}}}

\newcommand{\sbd}{\ensuremath{\:\!|\:\!}}
\newcommand{\wbd}{\ensuremath{\:{\text{\raisebox{1.5pt}{\rotatebox[origin=c]{90}{\Large\Kutline}}}}\:\!}}
\newcommand{\wbdo}{\ensuremath{\:{\text{\raisebox{1.5pt}{\rotatebox[origin=c]{90}{\Large\Kutline}}}}^{\!\!\mathrm{o}}\:\!}}
\usepackage{bbm}
\renewcommand{\ileft}[2][{}]{\ensuremath{\mathbbm{l}_{#1}(#2)}}
\renewcommand{\iright}[2][{}]{\ensuremath{\mathbbm{r}_{#1}(#2)}}

\begin{document}
\title{Hereditary Semiorders and Enumeration of Semiorders by Dimension}
\author{Mitchel T.\ Keller}
\address{Department of Natural and Mathematical Sciences\\Morningside
  College\\1501 Morningside Avenue\\Sioux City, IA 51106}
\email{kellerm@morningside.edu}
\author{Stephen J.\ Young}
\address{Pacific Northwest National Laboratory\\Richland, WA 99352}
\email{stephen.young@pnnl.gov}
\date{27 February 2020}
\thanks{\textit{PNNL Information Release:} PNNL-SA-130793}
\keywords{semiorder, ascent sequence, poset, dimension, enumeration,
  interval representation, (\twoplustwo)-free poset, generating
  functions, (\oneplusthree)-free poset}
\subjclass[2010]{06A07, 05A15}
\maketitle
\begin{abstract}
  In 2010, Bousquet-M\'elou et al. defined sequences of
  nonnegative integers called ascent sequences and showed that
  the ascent sequences of length $n$ are in one-to-one correspondence
  with the interval orders, i.e., the posets not containing the poset
  $\twoplustwo$. Through the use of generating
  functions, this provided an answer to the longstanding open question
  of enumerating the (unlabeled) interval orders. A semiorder is an
  interval order having a representation in which all intervals have
  the same length. In terms of forbidden subposets, the semiorders
  exclude \twoplustwo{} and \oneplusthree. The number of unlabeled
  semiorders on $n$ points has long been known to be the
  $n$\textsuperscript{th} Catalan number. However, describing the
  ascent sequences that correspond to the semiorders under the
  bijection of Bousquet-M\'elou et al.\ has proved difficult. In this
  paper, we discuss a major part of the difficulty in this area: the
  ascent sequence corresponding to a semiorder may have an initial
  subsequence that corresponds to an interval order that is not a
  semiorder.

  We define the hereditary semiorders to be those corresponding to an
  ascent sequence for which every initial subsequence also corresponds
  to a semiorder. We provide a structural result that characterizes
  the hereditary semiorders and use this characterization to determine
  the ordinary generating function for hereditary semiorders. We also
  use our characterization of hereditary semiorders and the
  characterization of semiorders of dimension $3$ given by Rabinovitch
  to provide a structural description of the semiorders of dimension
  at most $2$. From this description, we are able to determine the
  ordinary generating function for the semiorders of dimension at most $2$.
\end{abstract}
\section{Background and Motivation}

In this article, we investigate the bijective relationship between
interval orders and ascent sequences introduced by Bousquet-M\'elou et
al.\ in \cite{bousquet:intord-enum}. In that paper, the authors
answered a classic open question by providing an enumeration of
interval orders through a bijection with nonnegative sequences of
integers known as ascent sequences. The number of semiorders has long
been known to be given by the Catalan numbers, but no one has yet
given a description of the subclass of ascent sequences associated to
the semiorders by the bijection of Bousquet-M\'elou et al.\ in terms
of ascents. Most problematic is the fact that it is possible for an
ascent sequence to correspond to a semiorder while some initial
subsequence of that ascent sequence corresponds to an interval order
that is \emph{not} a semiorder. To address this, we define the class
of hereditary semiorders as those for which every initial subsequence
of the corresponding ascent sequence corresponds to a semiorder. The
hereditary semiorders can also be nicely described in terms of their
interval representation, and this structure further allows us to give
a characterization of the semiorders of dimension $2$ in terms of this
structure. Rabinovitch proved in \cite{rabinovitch:semiorder-dim}
that all semiorders have dimension at most $3$. Combined with the work
of Kelly in \cite{kelly:3-dim-irr} and Trotter and Moore in
\cite{trotter:3-dim-irr}, this led to a characterization of those of dimension $3$. Our analysis uses that characterization. With these structural results in hand, we are able to enumerate both the hereditary semiorders and the semiorders of dimension $2$.

\subsection*{Interval orders, semiorders, and dimension}

Before proceeding to our discussion of interval orders and semiorders,
we require a couple of definitions that apply to all posets. For a
poset $P = (X,\leq_P)$ and $x\in X$, the (open) \emph{down set} of
$x$, denoted by $\down{x}$, is $\set{y\in X\colon y <_P x}$. Dually,
the \emph{up set} of $x$, denoted by $\up{x}$ is $\set{y\in X\colon y
  >_P x}$. For a positive integer $n$, $\mbf{n}$ denotes the totally ordered poset with $n$ elements. If $n$ and $m$ are positive integers, then $\mbf{n}+\mbf{m}$ denotes the disjoint union of the posets $\mbf{n}$ and $\mbf{m}$. The posets $\twoplustwo$ and $\oneplusthree$ are depicted in Figure~\ref{fig:semi-forbidden}.

\begin{figure}[h]
  \centering
  \raisebox{-0.5\height}{\begin{tikzpicture}
    \draw [fill] (0,0) circle [radius=0.1];
    \draw [fill] (0,1) circle [radius=0.1];
    \draw [fill] (1,0) circle [radius=0.1];
    \draw [fill] (1,1) circle [radius=0.1];
    \draw (0,0) -- (0,1);
    \draw (1,0) -- (1,1);
    \node [below] at (0.5,-0.5) {\twoplustwo};
  \end{tikzpicture}}
  \hspace*{0.3\linewidth}
  \raisebox{-0.5\height}{\begin{tikzpicture}
    \draw [fill] (1,0) circle [radius=0.1];
    \draw [fill] (1,1) circle [radius=0.1];
    \draw [fill] (1,2) circle [radius=0.1];
    \draw [fill] (0,1) circle [radius=0.1];
    \draw (1,0) -- (1,1) -- (1,2);
    \node [below] at (0.5,-0.5) {\oneplusthree};
  \end{tikzpicture}}
  \caption{The posets \twoplustwo{} and \oneplusthree.}
  \label{fig:semi-forbidden}
\end{figure}

We call a poset $P = (X,\leq_P)$ an \emph{interval order} provided
that for each $x\in X$ there exists a closed, bounded interval $I(x) =
[\ileft{x},\iright{x}]$ of $\R$ such that $x <_P y$ if and only if $\iright{x} < \ileft{y}$, i.e., the interval of $x$ lies completely to the left of the interval of $y$. The collection of intervals associated to $P$ is called an \emph{interval representation} of $P$ (or just a \emph{representation}). An interval order $P$ is called a \emph{semiorder} provided that $P$ has an interval representation in which all intervals have the same (typically unit) length. The first appearance of what we today recognize as an interval order is in a paper by Wiener \cite{wiener:intords}. It wasn't until 1970, however, that the following theorem was established by Fishburn.
\begin{thm}[Fishburn \cite{fishburn:intords-char}]\label{thm:fishburn}
  Let $P=(X,\leq_P)$ be a poset. The following are equivalent:
  \begin{enumerate}
  \item $P$ is an interval order.
  \item $P$ does not contain $\twoplustwo$ as a subposet.
  \item If $x<_Py$ and $z<_Pw$, then $x<_P w$ or $z<_P y$.
  \item The collection of down sets of elements of $X$ is totally ordered by inclusion.
  \item The collection of up sets of elements of $X$ is totally ordered by inclusion.
  \end{enumerate}
\end{thm}

The characterization of semiorders was actually arrived at earlier in the form of a result in mathematical logic by Scott and Suppes.
\begin{thm}[Scott and Suppes \cite{scott:semiorders}]\label{thm:scott-suppes}
  A poset $P$ is a semiorder if and only if $P$ contains neither $\twoplustwo$ nor $\oneplusthree$ as a subposet.
\end{thm}

In \cite{greenough:intord-repn}, Greenough not only showed that when
$P=(X,\leq_P)$ is an interval order, the number of distinct down sets
of elements of $X$ is equal to the number of distinct up sets of
elements of $X$ but also gave an algorithm for generating a unique
interval representation using the smallest number of endpoints
possible. Although they did not discuss it in this manner, the
bijection of Bousquet-M\'elou et al.\ between ascent sequences and
interval orders (described in the next subsection) gives rise to such a representation, and such a representation will be central to our arguments. Thus, we briefly describe the algorithm and its critical properties here. To produce the representation, list the down sets of elements of $X$ as $D_0 \subsetneq D_1\subsetneq D_2\subsetneq\cdots\subsetneq D_{t-1}$, where $t$ is the number of distinct down sets (and hence up sets). Also list the up sets of elements of $X$ as $U_0\supsetneq U_1\supsetneq U_2\supsetneq \cdots\supsetneq U_{t-1}$. For $x\in X$, we define $I(x) = [i,j]$ where $\down{x} = D_i$ and $\up{x} = U_j$. Note that this may map distinct elements $x,y\in X$ to the same interval, which is allowed by our definition of interval representation. This happens if and only if $\down{x} = \down{y}$ and $\up{x} = \up{y}$. In this case, we say that $x$ and $y$ have \emph{duplicated holdings}. A poset in which no two elements have duplicated holdings is said to have \emph{no duplicated holdings}, sometimes abbreviated NODH.

In this article, we shall refer to the representation produced by the
algorithm described above as the \emph{minimal endpoint
  representation} of an interval order $P = (X,\leq_P)$. Because of
the manner in which the minimal endpoint representation is created, we
know that for each $i\in\set{0,\dots,t-1}$, there exist $x,y\in X$
such that $\ileft{x} = i$ and $\iright{y} = i$. That is, in a minimal
endpoint representation, every integer from $0$ to $t-1$ occurs as both a left endpoint and a right endpoint.

\begin{examp}
  To illustrate the algorithm for finding the minimal endpoint
  representation of an interval order, consider the poset shown in
  Figure~\ref{fig:intord-examp}. The down sets and up sets as ordered
  by the algorithm are listed below.
  \begin{align*}
    D_0 &= \empset & U_0 &= \set{b,c,d,y}\\
    D_1 &= \set{a} & U_1 &= \set{b,c,d}\\
    D_2 &=\set{a,x} & U_2 &=\set{c,d}\\
    D_3 &=\set{a,x,y} & U_3 &=\set{d}\\
    D_4 &=\set{a,c,x,y} &U_4 &=\empset
  \end{align*}
  Since $\down{x} = \empset$ and $\up{x} = \set{b,c,d}$, the algorithm tells us that in the minimal endpoint representation, $I(x) = [0,1]$ by locating the subscripts corresponding to these sets. Similarly, $\down{y} = \set{a}$ and $\up{y} = \set{c,d}$, so $I(y) = [1,2]$. The remaining four intervals of the minimal endpoint representation are found similarly, and the representation is depicted at the right in Figure~\ref{fig:intord-examp}.
  \begin{figure}[h]
    \centering
    \raisebox{-0.5\height}{\begin{tikzpicture}
      \draw (0,0) -- (0,1);
      \draw (1,0) -- (1,1) -- (1,2) -- (1,3);
      \draw (0,0) -- (1,2);
      \draw (0,1) -- (1,0);
      \draw [fill] (0,0) circle [radius=0.1];
      \draw [fill] (0,1) circle [radius=0.1];
      \draw [fill] (1,0) circle [radius=0.1];
      \draw [fill] (1,1) circle [radius=0.1];
      \draw [fill] (1,2) circle [radius=0.1];
      \draw [fill] (1,3) circle [radius=0.1];
      \node [left] at (0,0) {$x$};
      \node [left] at (0,1) {$b$};
      \node [right] at (1,0) {$a$};
      \node [right] at (1,1) {$y$};
      \node [right] at (1,2) {$c$};
      \node [right] at (1,3) {$d$};
    \end{tikzpicture}}
    \hspace*{0.3\linewidth}
    \raisebox{-0.5\height}{\begin{tikzpicture}
      \draw[{[-]}, thick] (0,0) -- (1,0);
      \draw[{[-]}, thick] (1,0.5) -- (2,0.5);
      \draw [thick] (0,0.37) -- (0,0.62);
      \draw [thick] (3,0.37) -- (3,0.62);
      \draw [thick] (4,0.37) -- (4,0.62);            
      \draw[{[-]}, thick] (2,0) -- (4,0);
      \node [below] at (0.5,0) {$x$};
      \node [below] at (3,0) {$b$};
      \node [above] at (0,0.62) {$a$};
      \node [above] at (1.5,0.62) {$y$};
      \node [above] at (3,0.62) {$c$};
      \node [above] at (4,0.62) {$d$};
    \end{tikzpicture}}
    \caption{An interval order and its minimal endpoint representation}
    \label{fig:intord-examp}
  \end{figure}
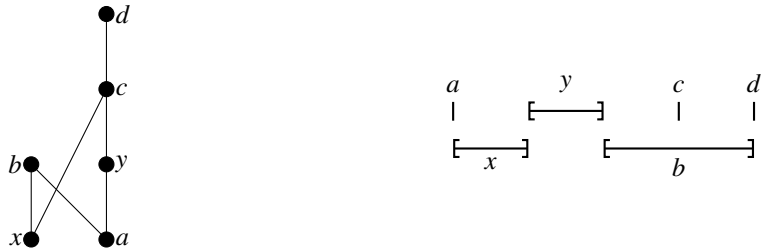
\end{examp}

Notice that when $P$ is a semiorder, its minimal endpoint representation is not necessarily one in which all intervals have the same length. The most straightforward example of this is $\mathbf{1}+\mathbf{2}$, which is shown in Figure~\ref{fig:1plus2-repn} along with its minimal endpoint representation.
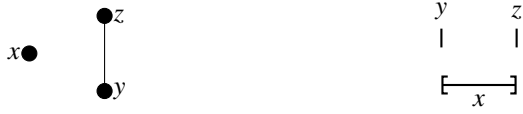
\begin{figure}[h]
  \centering
  \raisebox{-0.5\height}{\begin{tikzpicture}
    \draw [fill] (1,0) circle [radius=0.1];
    \draw [fill] (1,1) circle [radius=0.1];
    \draw [fill] (0,0.5) circle [radius=0.1];
    \node [left] at (0,0.5) {$x$};
    \node [right] at (1,0) {$y$};
    \node [right] at (1,1) {$z$};
    \draw (1,0) -- (1,1);
  \end{tikzpicture}}\hspace*{0.3\linewidth}
\raisebox{-0.5\height}{\begin{tikzpicture}
    \draw[{[-]}, thick] (0,0) -- (1,0);
    \draw[thick] (0,0.5) -- (0,0.75);
    \draw[thick] (1,0.5) -- (1,0.75);
    \node [below] at (0.5,0) {$x$};
    \node [above] at (0,0.75) {$y$};
    \node [above] at (1,0.75) {$z$};
  \end{tikzpicture}}
  \caption{The poset $\mbf{1}+\mbf{2}$ and its minimal endpoint representation}
  \label{fig:1plus2-repn}
\end{figure}

While the minimal endpoint representation of a semiorder does not have all intervals of the same length, there is a straightforward interval containment test to determine if a minimal endpoint representation of an interval order is one of a semiorder. We will frequently make use of the following lemma in this paper.
\begin{lem}\label{lem:min-repn-semi}
  An interval order $P$ is a semiorder if and only if its minimal endpoint representation does not include intervals $[a_1,b_1]$ and $[a_2,b_2]$ such that $a_1 < a_2$ and $b_2 < b_1$.
\end{lem}

\begin{proof}
  Let $P = (X,\leq_P)$ be an interval order. For the ``only if'' direction, we consider the contrapositive. The existence of intervals satisfying the conditions in the lemma means that the interval $[a_2,b_2]$ lies in the interior of $[a_1,b_1]$. Suppose that $I(x) = [a_1,b_1]$ and $I(z) = [a_2,b_2]$. Since the representation is minimal, we know that there are $y,w\in X$ such that $\iright{y} = a_1$ and $\ileft{w} = b_1$. Then $\set{x,y,z,w}$ is a \oneplusthree{} in $P$, so $P$ is not a semiorder.

  For the converse, suppose that $P$ is an interval order that is not a semiorder and let $\set{x,y,z,w}$ be a \oneplusthree{} in $P$ with $x$ incomparable to $y,z,w$ and $y < z < w$. Notice that \[\ileft{x}\leq \iright{y} < \ileft{z}\leq \iright{z} < \ileft{w}\leq \iright{x},\] which shows that $I(x)$ and $I(z)$ are the intervals we seek.
\end{proof}

The minimal endpoint representation of an interval
  order can also be encoded in matrix form. For NODH interval orders,
  Fishburn called these \emph{characteristic matrices} in
  \cite{fishburn:intordsbook}. They have been more recently studied by
Dukes and Parviainen in \cite{dukes:ascent-upper-triangular}; Dukes et al.\
in \cite{dukes:composition-matrices}; and
Jel\'inek in \cite{jelinek:catalan-pairs}. In
\cite{jelinek:catalan-pairs}, Jel\'inek studied the class of what he
calls \emph{Fishburn matrices} that extend to the case where
duplicated holdings are allowed. Our Lemma~\ref{lem:min-repn-semi} can
be recast in terms of matrices as in Proposition~16 of
\cite{dukes:composition-matrices}. Because the following work relies
on an understanding of the underlying minimal endpoint representation,
we choose not to further explore the matrix-based approach here.

If $P = (X,\leq_P)$ is a poset, we say that a total order $L$ on $X$ is a \emph{linear extension} of $P$ provided that for all $x,y\in X$, if $x\leq_P y$, then $x\leq y$ in $L$. The \emph{dimension} of $P$, denoted $\dim(P)$, is the least $d$ such that there exist linear extensions $L_1,L_2,\dots,L_d$ of $P$ such that (as sets of ordered pairs)
\begin{align*}\leq_P &= L_1\cap L_2\cap\cdots\cap L_d.\end{align*}
In \cite{bogart:canonical-intords}, Bogart et al.\ showed that for
every positive integer $d$, there exists an interval order having
dimension at least $d$. On the other hand, the situation for
semiorders is much more restricted. Rabinovitch showed in
\cite{rabinovitch:semiorder-dim} that if $P$ is a semiorder, then
$\dim(P)\leq 3$. Furthermore, $\dim(P) = 3$ if and only if $P$
contains one of the posets shown in
Figure~\ref{fig:semi-forb}. Rabinovitch's original version of this
result involved a limitation on the height of the semiorder. The independent work of Kelly in \cite{kelly:3-dim-irr} and Trotter
and Moore in \cite{trotter:3-dim-irr} provided a complete
characterization of the posets of dimension $3$. In light of their
results, Rabinovitch's three forbidden subposets for a semiorder to
have dimension at most $2$ was verified to be complete without
limitations as to height, as stated in
Corollary~3.3 of Trotter's monograph \cite{trotter:dimbook}.

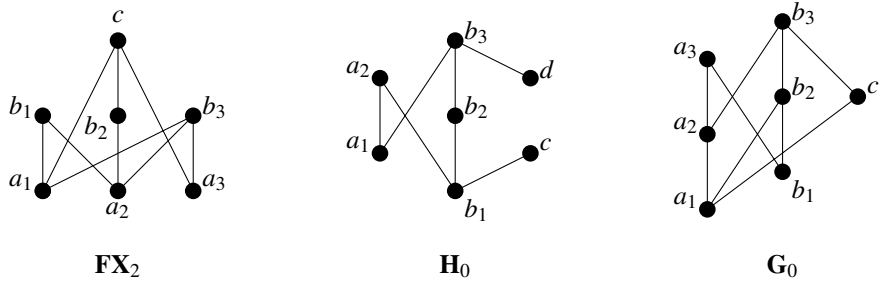
\begin{figure}[h]
  \centering
  \begin{tikzpicture}
    \draw (2,1) --(1,0) -- (1,1) -- (1,2);
    \draw (0,0) -- (0,1) -- (1,0);
    \draw (2,0) -- (2,1) -- (0,0) -- (1,2) -- (2,0);
    \draw [fill] (0,0) circle [radius=0.1];
    \node[vtx, label={[left] {$a_1$}}] at (0,0) {};
    \node[vtx, label={[left] {$b_1$}}] at (0,1) {};
    \node[vtx, label={[below, label distance = -0.15cm] {$a_2$}}] at (1,0) {};
    \node[vtx, label={[below left] {$b_2$}}] at (1,1) {};
    \node[vtx, label={[above] {$c$}}] at (1,2) {};
    \node[vtx, label={[right] {$a_3$}}] at (2,0) {};
    \node[vtx, label={[right] {$b_3$}}] at (2,1) {};
    \node at (1,-1) {$\mbf{FX}_2$};
  \end{tikzpicture}\hspace*{0.1\linewidth}
  \begin{tikzpicture}
    \draw (2,0.5) -- (1,0) -- (1,1) -- (1,2) -- (2,1.5);
    \draw (1,0) -- (0,1.5) -- (0,0.5) -- (1,2);
    \node [vtx, label={[left] {$a_1$}}] at (0,0.5) {};
    \node [vtx, label={[left] {$a_2$}}] at (0,1.5) {};
    \node [vtx, label={[below right, label distance=-0.1cm] {$b_1$}}] at (1,0) {};
    \node [vtx, label={[right] {$b_2$}}] at (1,1) {};
    \node [vtx, label={[right] {$b_3$}}] at (1,2) {};
    \node [vtx, label={[right] {$c$}}] at (2,0.5) {};
    \node [vtx, label={[right] {$d$}}] at (2,1.5) {};
    \node at (1,-1) {$\mbf{H}_0$};
  \end{tikzpicture}\hspace*{0.1\linewidth}
  \begin{tikzpicture}
    \draw (0,0) -- (0,1) -- (0,2) -- (1,0.5) -- (1,1.5) -- (1,2.5) -- (2,1.5) -- (0,0) -- (1,1.5);
    \draw (0,1) -- (1,2.5);
    \node [vtx, label={[left] {$a_1$}}] at (0,0) {};
    \node [vtx, label={[left] {$a_2$}}] at (0,1) {};
    \node [vtx, label={[left] {$a_3$}}] at (0,2) {};
    \node [vtx, label={[below right, label distance=-0.1cm] {$b_1$}}] at (1,0.5) {};
    \node [vtx, label={[right] {$b_2$}}] at (1,1.5) {};
    \node [vtx, label={[right] {$b_3$}}] at (1,2.5) {};
    \node [vtx, label={[right] {$c$}}] at (2,1.5) {};
    \node at (1,-0.75) {$\mbf{G}_0$};    
  \end{tikzpicture}
  \caption{The three subposets that can force a semiorder to have dimension $3$}
  \label{fig:semi-forb}
\end{figure}

For more information on interval orders and semiorders, Fishburn's monograph \cite{fishburn:intordsbook} is a classic while Trotter's survey article \cite{trotter:intords} provides a more recent look. The canonical work on dimension theory for posets is Trotter's monograph \cite{trotter:dimbook}. The labels on the posets in Figure~\ref{fig:semi-forb} follow Trotter's notation, but we note that there is an error in his list of the forbidden subposets for a semiorder to have dimension at most $2$. On page 197, he lists $\mbf{FX}_1$ in addition to the three given here, but $\mbf{FX}_1$ is not a semiorder since $\set{b_1,a_3,b_2,c}$ is a \oneplusthree.
\subsection*{Ascent sequences and enumeration}

Given a sequence $(x_1,x_2,\dots,x_i)$ of integers, the number of \emph{ascents} in the sequence is defined to be \[\asc(x_1,\dots,x_i) = \left|\set{1\leq j < i\colon x_j < x_{j+1}}\right|.\] In \cite{bousquet:intord-enum}, Bousquet-M\'elou et al.\ defined an \emph{ascent sequence} to be a sequence $(x_1,\dots,x_n)$ of nonnegative integers such that $x_1 = 0$ and $x_i\in [0,1+\asc(x_1,\dots,x_{i-1})]$ for all $2\leq i\leq n$. They then defined a map $\Psi$ from (unlabeled) interval orders on $n$ points to  ascent sequences of length $n$ and showed that their function is a bijection. Here, we recast the inverse of that bijection as a way to construct the minimal endpoint representation of an interval order on $n$ points from an ascent sequence of length $n$.

The process of constructing the interval order corresponding to an ascent sequence proceeds iteratively through the ascent sequence. The simplest ascent sequence, $(0)$, corresponds to the minimal endpoint representation $[0,0]$. To describe the algorithm, we assume that we have an ascent sequence $(x_1,\dots,x_n)$ with $n\geq 2$ and have constructed the interval order $Q$ corresponding to the ascent sequence $(x_1,\dots,x_{n-1})$. We retain some of the notation from \cite{bousquet:intord-enum} by letting $\lQ{Q}$ denote the \emph{largest right} endpoint of an interval in the minimal endpoint representation of $Q$. We also let $\lQ[*]{Q}$ denote the \emph{smallest left} endpoint of an interval \emph{with right endpoint} $\lQ{Q}$ (again, in the minimal endpoint representation of $Q$). Suppose now that $x_n = i$. We obtain the minimal endpoint representation of the interval order $P$ corresponding to $(x_1,\dots,x_n)$ by applying one of the following three moves:

\begin{itemize}
\item[\moveone] If $i \leq \lQ[*]{Q}$ add the interval $[i,\lQ{Q}].$
\item[\movetwo] If $i = \lQ{Q}+1$ add $[\lQ{Q}+1,\lQ{Q}+1]$.
\item[\movethree] If $\lQ[*]{Q} < i \leq \lQ{Q}$,
\begin{itemize}
  \item Replace every interval $[\lambda,\rho]$ for which $\lambda < i \leq \rho < \lQ{Q}$ with $[\lambda,\rho+1]$.
     \item Replace every interval
    $[\lambda,\rho]$ for which $i \leq \lambda \leq \rho \leq \lQ{Q}$ with
  $[\lambda+1,\rho+1]$.
  \item Replace every interval $[\lambda,\rho]$ for which $\lambda < i$ and $\rho = \lQ{Q}$ with $[\lambda,i]$.
\item Add the interval $[i, \lQ{Q}+1]$.
\end{itemize}
\end{itemize}

We know that the minimal endpoint representation of $Q$ contains $[\lQ[*]{Q},\lQ{Q}]$, and so \moveone{} adds another maximal element
to the poset whose interval extends at least as far left as
$\lQ[*]{Q}$ so that $\lQ[*]{P} = i$. If $i = \lQ[*]{Q}$, then \moveone{} merely adds another
point to the poset that has $[\lQ[*]{Q},\lQ{Q}]$ as its interval in the minimal endpoint
representation. This gives rise to a pair of points with duplicated
holdings. (This use of \moveone{} only occurs when $i = x_n =
x_{n-1}$, and this is the only way to create duplicated holdings. We
will frequently use this fact in our enumerative work later in the
paper.) \movetwo{} adds a new trivial interval that becomes the unique
maximal element in $P$. \movethree{} is the problematic move when it
comes to working with semiorders. Its effect is to increase the
largest right endpoint by one so that $\lQ{P} = \lQ{Q} + 1$ while
inserting a new endpoint at $i$. Any interval with its left endpoint
less than $i$ and its right endpoint at least $i$ has its right
endpoint moved one unit to the right. Any interval that had its left
endpoint being $i$ or larger is shifted to the right by one unit. Any
interval that corresponds to a maximal element in $Q$ (which is
equivalent to having right endpoint $\lQ{Q}$) is truncated by
retaining its left endpoint but making its right endpoint $i$. A new
interval $[i,\lQ{Q} + 1]$ is then inserted, which ensures that the
representation is minimal by having $i$ used as both a left and right
endpoint. We illustrate \movethree{} in Figure~\ref{fig:movethree} with intervals in the
    representation of $Q$ illustrated above and the corresponding
    intervals (in the same relative positions)
    and new interval (shown uppermost) of $P$ below.
  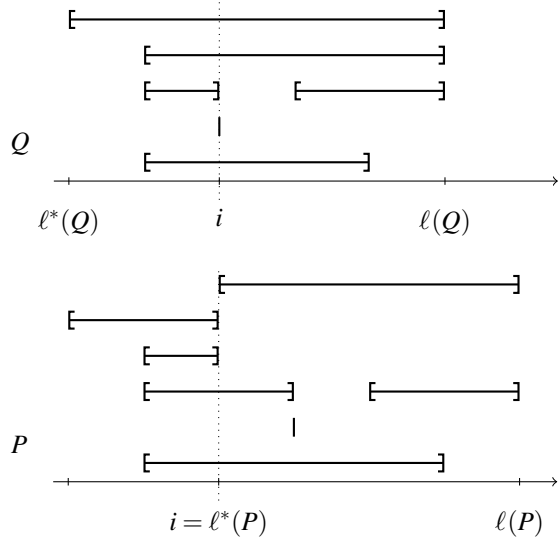
\begin{figure}[h]
    \centering
    \begin{tikzpicture}
      \node at (-0.65,0.25) {$Q$};
    \draw[thick] (2,0.35) -- (2,0.6);
    \draw[{[-]}, thick] (1,0) -- (4,0);
    \draw[{[-]}, thick] (1,0.95) -- (2,0.95);
    \draw[{[-]}, thick] (3,0.95) -- (5,0.95);
    \draw[{[-]}, thick] (1,1.425) -- (5,1.425);
    \draw[{[-]}, thick] (0,1.9) -- (5,1.9);
    \draw[->] (-0.2,-0.25) -- (6.5,-0.25);
    \draw[dotted] (2,-0.5) -- (2,2.2);
    \foreach \x in {0,2,5}
       \draw (\x,-0.3) -- (\x,-0.2);
       \node [below] at (0,-0.5) {$\lQ[*]{Q}$};
           \node [below] at (2,-0.5) {$i$};
    \node [below] at (5,-0.5) {$\lQ{Q}$};
  \end{tikzpicture}\\\vspace*{2ex}
      \begin{tikzpicture}
      \node at (-0.65,0.25) {$P$};
    \draw[thick] (3,0.35) -- (3,0.6);
    \draw[{[-]}, thick] (1,0) -- (5,0);
    \draw[{[-]}, thick] (1,0.95) -- (3,0.95);
    \draw[{[-]}, thick] (4,0.95) -- (6,0.95);
    \draw[{[-]}, thick] (1,1.425) -- (2,1.425);
    \draw[{[-]}, thick] (0,1.9) -- (2,1.9);
    \draw[{[-]}, thick] (2,2.375) -- (6,2.375);
        \draw[dotted] (2,-0.5) -- (2,2.55);
    \draw[->] (-0.2,-0.25) -- (6.5,-0.25);
    \foreach \x in {0,2,6}
       \draw (\x,-0.3) -- (\x,-0.2);
           \node [below] at (2,-0.5) {$i=\lQ[*]{P}$};
    \node [below] at (6,-0.5) {$\lQ{P}$};
     \end{tikzpicture}
  \caption{The effect of a \movethree{}}
    \label{fig:movethree}
  \end{figure}

\begin{examp}
  We illustrate the process of constructing the minimal endpoint representation of the interval order corresponding to the ascent sequence $(0,1,2,3,1,0,1,3)$. We will denote by $Q_i$ the interval order corresponding to the first $i$ terms of the given ascent sequence. We know that $Q_1$ is represented by $\set{[0,0]}$ and $\lQ{Q_1} = \lQ[*]{Q_1} = 0$. Thus, $Q_2$ is represented by $\set{[0,0],[1,1]}$ through a \movetwo. The next two moves are also \movetwo, which leads us to $Q_4$ being represented by $\set{[0,0],[1,1],[2,2],[3,3]}$. We have $\lQ{Q_4} = \lQ[*]{Q_4} = 3$. Thus, to form $Q_5$, we apply \moveone, which adds the interval $[1,3]$, giving the representation illustrated in Figure~\ref{fig:moves-examp}. (At each stage, we place a $\bullet$ above the new interval added at that stage.) 
  \begin{figure}[h]
    \centering
    \begin{tikzpicture}
      \node at (-0.65,0.25) {$Q_5$};
    \draw[thick] (0,0.5) -- (0,0.75);
    \draw[thick] (1,0.5) -- (1,0.75);
    \draw[thick] (2,0.5) -- (2,0.75);
    \draw[thick] (3,0.5) -- (3,0.75);
    \draw[{[-]}, thick] (1,0) -- (3,0);
    \node [above] at (2,0) {$\bullet$};
    \draw[->] (-0.2,-0.25) -- (3.5,-0.25);
    \foreach \x in {0,1,2,3}
    \draw (\x,-0.3) -- (\x,-0.2) node [below] {\scriptsize\x};
    \node [below] at (1,-0.5) {$\lQ[*]{Q_5}$};
    \node [below] at (3,-0.5) {$\lQ{Q_5}$};
  \end{tikzpicture}\hspace*{0.15\linewidth}
    \begin{tikzpicture}
      \node at (-0.65,0.25) {$Q_6$};
    \draw[thick] (0,0.5) -- (0,0.75);
    \draw[thick] (1,0.5) -- (1,0.75);
    \draw[thick] (2,0.5) -- (2,0.75);
    \draw[thick] (3,0.5) -- (3,0.75);
    \draw[{[-]}, thick] (1,0) -- (3,0);
    \draw[{[-]}, thick] (0,1.1) -- (3,1.1);
    \node [above] at (1.5,1.1) {$\bullet$};
    \draw[->] (-0.2,-0.25) -- (3.5,-0.25);
    \foreach \x in {0,1,2,3}
    \draw (\x,-0.3) -- (\x,-0.2) node [below] {\scriptsize\x};
    \node [below] at (0,-0.5) {$\lQ[*]{Q_6}$};
    \node [below] at (3,-0.5) {$\lQ{Q_6}$};
     \end{tikzpicture}\\
     \vspace{20pt}
    \begin{tikzpicture}
      \node at (-0.65,0.25) {$Q_7$};
    \draw[thick] (0,0.5) -- (0,0.75);
    \draw[thick] (2,0.5) -- (2,0.75);
    \draw[thick] (3,0.5) -- (3,0.75);
    \draw[thick] (4,0.5) -- (4,0.75);
    \draw[{[-]}, thick] (2,0) -- (4,0);
    \draw[{[-]}, thick] (0,1.1) -- (1,1.1);
    \draw[{[-]}, thick] (1,1.5) -- (4,1.5);
    \node [above] at (2.5,1.5) {$\bullet$};
    \draw[->] (-0.2,-0.25) -- (4.5,-0.25);
    \foreach \x in {0,1,2,3,4}
       \draw (\x,-0.3) -- (\x,-0.2) node [below] {\scriptsize\x};
    \node [below] at (1,-0.5) {$\lQ[*]{Q_7}$};
    \node [below] at (4,-0.5) {$\lQ{Q_7}$};
     \end{tikzpicture}
  \hspace*{0.03\linewidth}
  \begin{tikzpicture}
      \node at (-0.65,0.25) {$Q_8$};
    \draw[thick] (0,0.5) -- (0,0.75);
    \draw[thick] (2,0.5) -- (2,0.75);
    \draw[thick] (4,0.5) -- (4,0.75);
    \draw[thick] (5,0.5) -- (5,0.75);
    \draw[{[-]}, thick] (2,0) -- (3,0);
    \draw[{[-]}, thick] (0,1.1) -- (1,1.1);
    \draw[{[-]}, thick] (1,1.5) -- (3,1.5);
    \draw[{[-]}, thick] (3,1.1) -- (5,1.1);
    \node [above] at (4,1.1) {$\bullet$};
    \draw[->] (-0.2,-0.25) -- (5.2,-0.25);
    \foreach \x in {0,1,2,3,4,5}
    \draw (\x,-0.3) -- (\x,-0.2) node [below] {\scriptsize\x};
    \node [below] at (3,-0.5) {$\lQ[*]{Q_8}$};
    \node [below] at (5,-0.5) {$\lQ{Q_8}$};
  \end{tikzpicture}
  \caption{Constructing the minimal endpoint representation of the interval order corresponding to ascent sequence $(0,1,2,3,1,0,1,3)$}
    \label{fig:moves-examp}
  \end{figure}
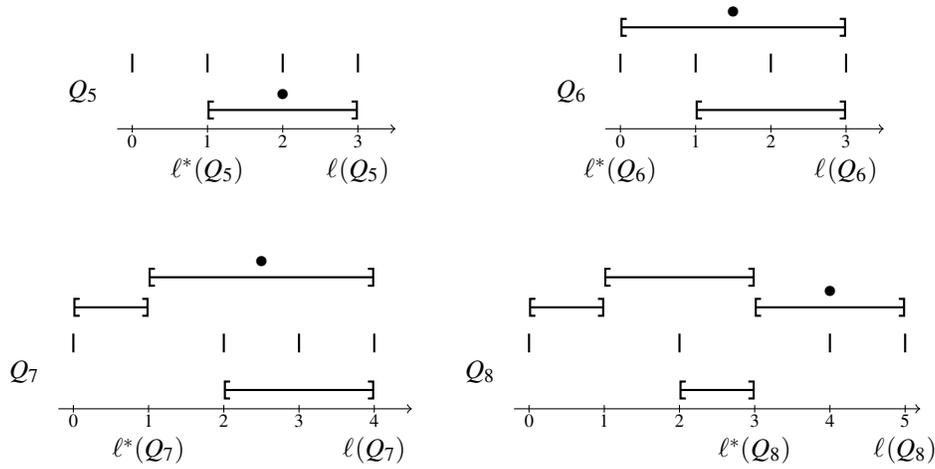
We also form $Q_6$ by using \moveone{} and have $\lQ{Q_6} = 3$ and $\lQ[*]{Q_6} = 0$. This means that $Q_7$ must be formed by using \movethree, as depicted in the figure. The interval of length two and the three trivial intervals are shifted right, while the interval $[0,3]$ is truncated to become $[0,1]$ and the new interval is $[1,4]$. This leaves us with $\lQ{Q_7} = 4$ and $\lQ[*]{Q_7} = 1$, so finishing requires another \movethree. This \movethree{} shifts two intervals, stretches no intervals, and truncates two intervals.
\end{examp}

Since the work of Bousquet-M\'elou et al., a variety of results
building on their work have been published. Many of them relate to
pattern-avoiding permutations and specialized classes of ascent
sequences. However, it is worth highlighting some of those with
connections to posets. For instance, in \cite{kitaev:intord-stats},
Kitaev and Remmel enumerated interval orders by number of minimal
elements (and other statistics). They also identified a subset of the
ascent sequences that they termed the \emph{restricted ascent
  sequences} and showed that the number of ascent sequences of length
$n$ is enumerated by the $n$\textsuperscript{th} Catalan
number. However, the bijection between ascent sequences and interval
orders does not send the restricted ascent sequences to the
semiorders, and the authors were unable to characterize the interval
orders corresponding to  the restricted ascent sequences. They did conjecture a refined version of their generating function for enumeration by number of minimal elements, which was proved independently by Levande in \cite{levande:fishburn-diagrams} and Yan in \cite{yan:intord-enum}. Dukes et al.\ looked at enumeration by the number of indistinguishable elements in \cite{dukes:intord-indist}, while Khamis enumerated the number of interval orders with no duplicated holdings by height in \cite{khamis:intord-height}. The focus of Jel\'inek's work in \cite{jelinek:self-dual} was to enumerate the number of self-dual interval orders. Claesson and Linusson looked at connections between various classes of matchings and interval orders in \cite{claesson:matchings-posets}. Disanto et al.\ looked at some problems involving generating and enumerating series parallel interval orders and semiorders in \cite{disanto:intords}.

As mentioned earlier, the number of semiorders on $n$ points is the
$n$\textsuperscript{th} Catalan number, as shown by Wine and Freund in
\cite{wine:semiorder-enum} and Dean and G.\ Keller in
\cite{dean:semiorder-enum}. Greenough showed in
\cite{greenough:intord-repn} that the number of semiorders on $n$
points with no
duplicated holdings is given by
\[s(n) =
  \sum_{a=0}^{\lfloor\frac{n-1}{2}\rfloor}\binom{n-1}{a,n-1-2a,a} -
  \sum_{a=2}^{\lfloor\frac{n+1}{2}\rfloor}\binom{n-1}{a,n+1-2a,a-2},\]
where the terms of the sums are multinomial coefficients.  More
recently, Lewis and Zhang were able to enumerate the number of graded
posets (not interval orders) that do not contain \oneplusthree{} in
\cite{lewis:graded-no-oneplusthree}, which was followed by the work of
Guay-Paquet et al.\ in \cite{guay-paquet:nothreeplusone}. An
enumeration of semiorders by length (one less than the number of
elements in a maximum chain) was given by Hu in
\cite{hu:semiorders-height}. The only work we are aware of that
enumerates any class of posets by dimension is the work of El-Zahar
and Sauer in \cite{elzahar-enum-2-dim}, where they provide an
asymptotic enumeration of two-dimensional posets. Our work in the
remainder of the paper will be restricted to unlabeled semiorders. We
will proceed to define what we call the hereditary semiorders and
characterize their structure. This structure then gives a way to
access the semiorders of dimension $2$. Our enumeration will proceed
by looking at the ascent sequences corresponding to these classes of
semiorders. The sequences of integers produced as a result were not in
OEIS prior to this work.

\section{Block Structure of Hereditary Semiorders}

It is straightforward to verify that the ascent sequence $(0, 1, 0, 1, 2, 0)$ has minimal endpoint representation as an interval order as shown in Figure~\ref{fig:nonhered1} under the bijection $\Psi\inv$ of Bousquet-M\'elou et al. By Lemma~\ref{lem:min-repn-semi}, we can tell that this is not a semiorder.
\begin{figure}[h]
  \centering
    \begin{tikzpicture}
    \draw[thick] (0,0.35) -- (0,0.6);
    \draw[thick] (3,0.35) -- (3,0.6);
    \draw[{[-]}, thick] (0,0) -- (3,0);
    \draw[{[-]}, thick] (1,0.475) -- (2,0.475);
    \draw[{[-]}, thick] (0,0.95) -- (1,0.95);
    \draw[{[-]}, thick] (2,0.95) -- (3,0.95);

    \draw[->] (-0.2,-0.25) -- (3.2,-0.25);
    \foreach \x in {0,1,2,3}
       \draw (\x,-0.3) -- (\x,-0.2) node [below] {\scriptsize\x};
  \end{tikzpicture}
  \caption{The interval order corresponding to $(0,1,0,1,2,0)$}
  \label{fig:nonhered1}
\end{figure}
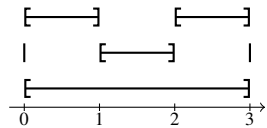
Progressing to the ascent sequence $(0, 1, 0, 1, 2, 0,2)$ requires a \movethree, however, which destroys the \oneplusthree, giving us the minimal endpoint representation depicted in Figure~\ref{fig:nonhered2}. Since no interval is contained in the interior of any other interval, we know that this \emph{is} a semiorder. This dilemma leads us to make the following definition.
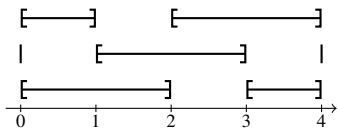
\begin{figure}[h]
  \centering
    \begin{tikzpicture}
    \draw[thick] (0,0.35) -- (0,0.6);
    \draw[thick] (4,0.35) -- (4,0.6);
    \draw[{[-]}, thick] (0,0) -- (2,0);
    \draw[{[-]}, thick] (1,0.475) -- (3,0.475);
    \draw[{[-]}, thick] (0,0.95) -- (1,0.95);
    \draw[{[-]}, thick] (3,0) -- (4,0);
    \draw[{[-]}, thick] (2,0.95) -- (4,0.95);
    \draw[->] (-0.2,-0.25) -- (4.2,-0.25);
    \foreach \x in {0,1,2,3,4}
       \draw (\x,-0.3) -- (\x,-0.2) node [below] {\scriptsize\x};
  \end{tikzpicture}
  \caption{The semiorder corresponding to $(0,1,0,1,2,0,2)$}
  \label{fig:nonhered2}
\end{figure}

\begin{defn}\label{defn:hereditary}
  Let $P$ be a semiorder on $n$ points, and let $(x_1,\dots,x_n)$ be the ascent sequence corresponding to $P$ under $\Psi$. We say that $P$ is \emph{hereditary} provided that for every $i$ with $1\leq i\leq n$, $\Psi\inv((x_1,\dots,x_i))$ is a semiorder.
\end{defn}

As we will show, the hereditary semiorders have a particularly nice structure in terms of their minimal endpoint representations. We describe this structure as being built from certain fundamental blocks with three types of boundary options for how different blocks can be combined to form a larger hereditary semiorder.

\begin{defn}\label{defn:blocks}
  The fundamental blocks we use to characterize the hereditary semiorders are as given below. Throughout, $b$ is a nonnegative integer and $k$ is an integer.
  \begin{align*}
    T_0 &= \set{[0,0]}\\
    T_1^b &= \set{[b,b],[b+1,b+1]}\\
    W_k^b &=\set{[b,b],[b+k,b+k]}\cup\bigcup_{i=0}^{k-1}\set{[b+i,b+i+1]}\quad\text{for }k\geq 1\\
    C_2^b &= \set{[b,b],[b,b+1],[b,b+2],[b+1,b+2],[b+2,b+2]}\\
    U_k^b &= \bigcup_{i=0}^{k-1} \set{[b,b+i],[b+k-i,b+k]}\quad\text{for }k\geq 3\\
    C_k^b &= \set{[b,b+k]}\cup U_k^b\quad\text{for }k\geq 3
  \end{align*}
  We refer to $T_0$ as the \emph{trivial block}. A \emph{nontrivial
    block} is any block that is not $T_0$. We will occasionally omit
  the superscript and refer to $T_1$ if the position of the block is clear from context. If we wish to refer to a generic block of the form $W_k^b$, we will use $\Wb$. For $U_k^b$, we will write $\Ub$, and for $C_k^b$ we will write $\Cb$. By $\Bb$, we will mean a block that could either be a $\Cb$ or a $\Ub$.
\end{defn}

We give sample illustrations of some of the blocks defined in
Definition~\ref{defn:blocks} in Figure~\ref{fig:sample-blocks}.
\begin{figure}[h]
  \centering
      \begin{tikzpicture}
    \draw[thick] (0,-0.125) -- (0,0.125);
    \draw[{[-]}, thick] (1,0) -- (4,0);

    \draw[{[-]}, thick] (0,0.475) -- (1,0.475);
    \draw[{[-]}, thick] (2,0.475) -- (4,0.475);
    
    \draw[{[-]}, thick] (0,0.95) -- (2,0.95);
    \draw[{[-]}, thick] (3,0.95) -- (4,0.95);

    \draw[{[-]}, thick] (0,1.425) -- (3,1.425); 
    \draw[thick] (4,1.3) -- (4,1.55);

    \draw[->] (-0.2,-0.25) -- (4.2,-0.25);
    \foreach \x in {0,1,2,3,4}
    \draw (\x,-0.3) -- (\x,-0.2) node [below] {\scriptsize\x};

    \node [below] at (2,-0.7) {$U_4^0$};
  \end{tikzpicture}\hspace*{0.1\linewidth}
      \begin{tikzpicture}
    \draw[thick] (0,-0.125) -- (0,0.125);
    \draw[{[-]}, thick] (1,0) -- (4,0);

    \draw[{[-]}, thick] (0,0.475) -- (1,0.475);
    \draw[{[-]}, thick] (2,0.475) -- (4,0.475);
    
    \draw[{[-]}, thick] (0,0.95) -- (2,0.95);
    \draw[{[-]}, thick] (3,0.95) -- (4,0.95);

    \draw[{[-]}, thick] (0,1.425) -- (3,1.425); 
    \draw[thick] (4,1.3) -- (4,1.55);

    \draw[{[-]}, thick] (0,1.9) -- (4,1.9); 
    \draw[->] (-0.2,-0.25) -- (4.2,-0.25);
    \foreach \x in {0,1,2,3,4}
    \draw (\x,-0.3) -- (\x,-0.2) node [below] {\scriptsize\x};

    \node [below] at (2,-0.7) {$C_4^0$};
  \end{tikzpicture}\\%
      \begin{tikzpicture}
    \draw[thick] (0,-0.125) -- (0,0.125);
    \draw[{[-]}, thick] (1,0) -- (2,0);

    \draw[{[-]}, thick] (0,0.475) -- (1,0.475);
    \draw[{[-]}, thick] (2,0.475) -- (3,0.475);

    \draw[thick] (3,-0.125) -- (3,0.125);

    \draw[->] (-0.2,-0.25) -- (3.2,-0.25);
    \foreach \x in {0,1,2,3}
    \draw (\x,-0.3) -- (\x,-0.2) node [below] {\scriptsize\x};

    \node [below] at (1.5,-0.7) {$W_3^0$};
  \end{tikzpicture}
  \caption{The blocks $U_4^0$, $C_4^0$, and $W_3^0$}
  \label{fig:sample-blocks}
\end{figure}
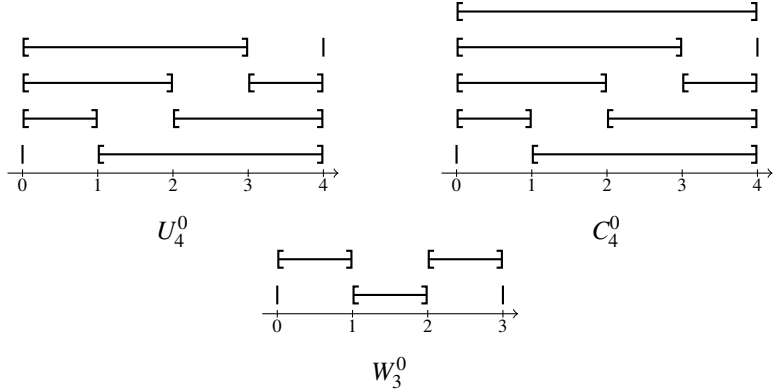

Our definitions of the blocks is in terms of sets, but we have already discussed the fact that a semiorder with duplicated holdings will have multiple elements associated to the same interval. Because we know how duplicated holdings arise in terms of the ascent sequence, we will be able to disregard such issues in terms of the block structure, stating our results in terms of the intervals appearing in the interval representation and implicitly allowing multiple points of the semiorder to have the same interval associated. We will, however, be able to readily address duplicated holdings when we get to our enumerative results later in the paper.

Next we define ways in which blocks can be combined. Note that the sum
of the subscript and superscript on any block gives the largest
endpoint of an interval in the block (and the superscript is the
smallest endpoint).

\begin{defn}\label{defn:boundaries}
  Let $A_{k_1}^{b}$ and $A_{k_2}^{k_1+b}$ be nontrivial blocks. We
  combine $A_{k_1}^b$ and $A_{k_2}^{k_1+b}$ with a \emph{strong
    boundary}, denoted by $A_{k_1}^b\sbd A_{k_2}^{k_1+b}$, by taking
  the set-theoretic union of the two blocks. The intersection between
  these two blocks is only the interval $[b+k_1,b+k_1]$. We can join
  two blocks (neither $T_1^b$ and not both $\Wb$) with a \emph{weak
    boundary}, denoted by $A_{k_1}^b\wbd A_{k_2}^{k_1+b}$, by removing
  $[b+k_1,b+k_1]$ from $A_{k_1}^b\cup A_{k_2}^{k_1+b}$. When
  $A_{k_1}^b$ is a $\Cb$ or $\Ub$ and $A_{k_2}^{k_1+b}$ is a $\Wb$ or
  $C_2^{k_1+b}$, we must also allow a \emph{weak boundary with
    optional element} (or \emph{optional interval}), which we denote
  by $A_{k_1}^b\wbdo A_{k_2}^{k_1+b}$. The intervals in
  $A_{k_1}^b\wbdo A_{k_2}^{k_1+b}$ are
  the same as with $A_{k_1}^b \wbd A_{k_2}^{k_1+b}$ with the addition
  of the interval $[k_1+b-1,k_1+b+1].$
\end{defn}

Note that weak boundaries are not permitted when one of the blocks is $T_1^b$,
since we would not be left with a minimal endpoint representation. We also do
not permit a weak boundary between two $\Wb$, since such a
construction would simply produce a $\Wb$ with larger subscript. From the definitions alone, it is not clear that it is sufficient to define a weak boundary with optional element only in the restricted cases given in Definition~\ref{defn:boundaries}. However, our argument will show that such a boundary cannot occur elsewhere. The block $T_0$ exists only to account for the antichain poset in which no two distinct points are comparable to one another, and $T_0$ cannot be combined with other blocks.

\begin{examp}
  In Figure~\ref{fig:blocks-examp}, we now illustrate the intervals of a semiorder with block structure
  \[C_3^0\wbdo W_2^3\sbd U_3^5 \wbd W_1^8 \sbd T_1^9.\]
  Notice that the interval $[3,3]$, which would be present in $C_3^0$, is absent because of the weak boundary. The interval $[8,8]$ is also omitted because of a weak boundary. The interval $[2,4]$ that bridges the boundary between the first two blocks is the optional interval.
  \begin{figure}[h]
    \centering
    \begin{tikzpicture}
      \draw[thick] (0,0.35) -- (0,0.6);
      \draw[thick] (9,0.35) -- (9,0.6);
      \draw[thick] (10,0.35) -- (10,0.6);
      \draw[{[-]}, thick] (1,0.475) -- (3,0.475);
      \draw[{[-]}, thick] (5,0.475) -- (6,0.475);
      \draw[{[-]}, thick] (7,0.475) -- (8,0.475);
      
      \draw[{[-]}, thick] (0,0) -- (1,0);
      \draw[{[-]}, thick] (2,0) -- (3,0);
      \draw[{[-]}, thick] (4,0) -- (5,0);
      \draw[{[-]}, thick] (8,0) -- (9,0); 

      \draw[{[-]}, thick] (0,0.95) -- (2,0.95);
      \draw[{[-]}, thick] (3,0.95) -- (4,0.95);
      \draw[{[-]}, thick] (6,0.95) -- (8,0.95);      
      \draw[thick] (5,0.825) -- (5,1.075);      

      \draw[{[-]}, thick] (0,1.425) -- (3,1.425);
      \draw[{[-]}, thick] (5,1.425) -- (7,1.425);

      \draw[{[-]}, thick] (2,1.9) -- (4,1.9);

      \foreach \bdry in {3,5,8,9}
      \draw[dotted] (\bdry,0) -- (\bdry,2.5);

      \node [above] at (1.5,2.1) {$C_3^0$};
      \node [above] at (4,2.1) {$W_2^3$};
      \node [above] at (6.5,2.1) {$U_3^5$};
      \node [above] at (8.5,2.1) {$W_1^8$};
      \node [above] at (9.5,2.1) {$T_1^9$};
      
      \draw[->] (-0.2,-0.25) -- (10.2,-0.25);
      \foreach \x in {0,1,2,3,4,5,6,7,8,9,10}
      \draw (\x,-0.3) -- (\x,-0.2) node [below] {\scriptsize\x};
    \end{tikzpicture}
    \caption{A sample block decomposition}
    \label{fig:blocks-examp}
  \end{figure}
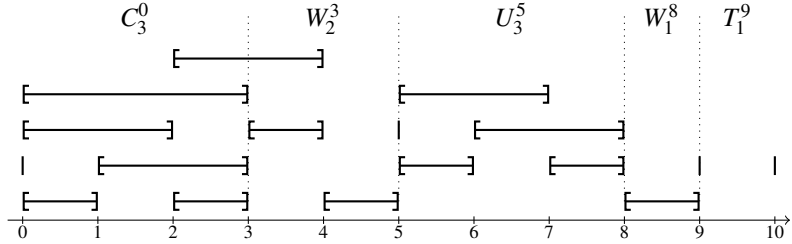
\end{examp}

We are now ready to state and prove our result about the structure of hereditary semiorders.

\begin{thm}\label{thm:block-structure}
  If $P$ is a hereditary semiorder, then intervals in the minimal endpoint representation of $P$ can be uniquely described using the blocks of Definition~\ref{defn:blocks} combined with the boundaries of Definition~\ref{defn:boundaries}. 
\end{thm}

\begin{proof}
  Our proof is by induction on $n$, the number of points in $P$. Since
  points with duplicated holdings do not impact the intervals in the
  minimal endpoint representation, we will assume without loss of
  generality that $P$ has no duplicated holdings. For $n=1$, the only
  option is a single point, which has minimal endpoint representation
  of the interval $[0,0]$. This is $T_0$. Now suppose for some
  positive integer $n$ that if $Q$ is a hereditary semiorder on $n$
  points, then $Q$ can be  described in terms of blocks and
  boundaries. Let $P$ be a hereditary semiorder on $n+1$ points, and
  let $(x_1,\dots,x_n,x_{n+1})$ be the corresponding ascent sequence
  under the bijection $\Psi$ of Bousquet-M\'elou et al. Since $P$ is hereditary, we know that $\Psi\inv((x_1,\dots,x_n))$ is a semiorder $Q$ on $n$ points. Therefore, by the induction hypothesis, the intervals in the minimal endpoint representation of $Q$ can be described in terms of blocks and boundaries. The proof is by cases based first on the last boundary and block in the block structure of $Q$ and second on the value of $\alpha = x_{n+1}$.

  When the last block is $T_1^b$ (and hence $\lQ{Q} = \lQ[*]{Q} =
  b+1$), the last boundary must be strong by definition. If $\alpha
  \leq b-1$, then  $[b,b]$ lies in the interior of the new interval
  added to form $P$, and so $P$ is not a semiorder by
  Lemma~\ref{lem:min-repn-semi}. When $\alpha = b$, we add the
  interval $[b,b+1]$, and thus the last boundary and block changes
  from $\sbd T_1^b$ to $\sbd W_1^b$. The case $\alpha = b+1$ results
  in duplicated holdings. Finally, when $\alpha = b+2$, a \movetwo{}
  is used and the block structure of $P$ ends in $\sbd T_1^b\sbd T_1^{b+1}$.

Before getting into the details of the other blocks and boundaries,
note that when the final block has subscript $a$ and superscript $b$,
taking $\alpha = a+b+1$ always results in a \movetwo{} that adds the trivial interval $[b+a+1,b+a+1]$. This adds $\sbd T_1^{b+a}$ to the end of the block structure of $Q$ to form the block structure of $P$. Therefore, we will not consider this situation below.

  We now consider when the block structure of $Q$ ends $\sbd C_a^b$,
  which implies $\lQ[*]{Q} = b$. When $\alpha = b$, the result is duplicated holdings. When $\alpha < b$, \moveone{} is used, adding the interval $[\alpha, a+b]$. This places the interval $[b,b]$, which exists because of the strong boundary, in the interior of the new interval, and so $P$ is not a semiorder. When $\alpha$ satisfies $b+1\leq \alpha \leq a+b$, a \movethree{} is applied to construct $P$. For $\alpha = b+1$, this converts the $\sbd C_a^b$ at the end of the block structure of $Q$ into $\sbd U_{a+1}^b$ as depicted in Figure~\ref{fig:strong-C-to-U}. When $\alpha$ satisfies $b+2\leq \alpha< b+a$, the \movethree{} results in an interval order that is not a semiorder. This is because \movethree{} extends the interval $[b,\alpha]$ in $Q$ to the interval $[b,\alpha+1]$ in $P$ and truncates the interval $[b+1,a+b]$ in $Q$ to become the interval $[b+1,\alpha]$ in $P$. This results in one interval contained in the interior of another, violating Lemma~\ref{lem:min-repn-semi}. When $\alpha = b+a$, no intervals are truncated but the interval $[b+a,b+a]$ becomes $[b+a+1,b+a+1]$ and the interval $[b+a,b+a+1]$ is added. This results in the block structure of $P$ ending $C_a^b\wbd W_1^{b+a}$.

  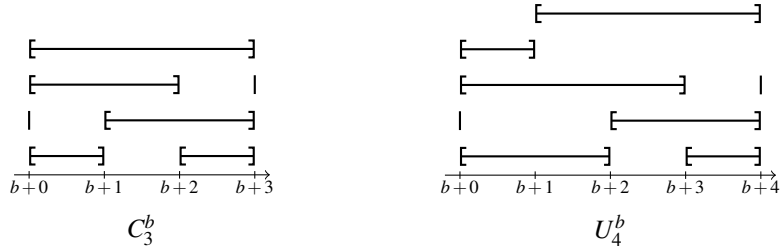
\begin{figure}[h]
    \centering
    \begin{tikzpicture}
      \draw[thick] (0,0.35) -- (0,0.6);
      \draw[thick] (3,0.825) -- (3,1.075);
       \draw[{[-]}, thick] (1,0.475) -- (3,0.475);
      
       \draw[{[-]}, thick] (0,0) -- (1,0);
       \draw[{[-]}, thick] (2,0) -- (3,0);

       \draw[{[-]}, thick] (0,0.95) -- (2,0.95);

       \draw[{[-]}, thick] (0,1.425) -- (3,1.425);
      
      \draw[->] (-0.2,-0.25) -- (3.2,-0.25);
      \foreach \x in {0,1,2,3}
      \draw (\x,-0.3) -- (\x,-0.2) node [below] {\scriptsize $b+\x$};
      \node [below] at (1.5,-.65) {$C_3^b$};
    \end{tikzpicture}\hspace*{0.15\linewidth}
    \begin{tikzpicture}
      \draw[thick] (0,0.35) -- (0,0.6);
      \draw[thick] (4,0.825) -- (4,1.075);
       \draw[{[-]}, thick] (2,0.475) -- (4,0.475);
      
       \draw[{[-]}, thick] (0,0) -- (2,0);
       \draw[{[-]}, thick] (3,0) -- (4,0);

       \draw[{[-]}, thick] (0,0.95) -- (3,0.95);

       \draw[{[-]}, thick] (0,1.425) -- (1,1.425);
       \draw[{[-]}, thick] (1,1.9) -- (4,1.9);
       \draw[->] (-0.2,-0.25) -- (4.2,-0.25);
      \foreach \x in {0,1,2,3,4}
      \draw (\x,-0.3) -- (\x,-0.2) node [below] {\scriptsize $b+\x$};
      \node [below] at (2,-.65) {$U_4^b$};
    \end{tikzpicture} 
    \caption{$\alpha = b+1$ when $Q$ ends $\sbd C_3^b$}
    \label{fig:strong-C-to-U}
  \end{figure}

  The next case is that the block structure of $Q$ ends $\sbd U_a^b$,
  which means $\lQ[*]{Q} = b+1$. When $\alpha < b$, we use a \moveone,
  which adds the interval $[\alpha,b+a]$. This interval contains
  $[b,b+1]$ in its interior, and so $P$ would not be a semiorder. For
  $\alpha = b$, this adds the interval required to convert the $U_a^b$
  to a $C_a^b$ while retaining the strong boundary. For $\alpha = b+1$, we produce duplicated holdings. The remaining cases involve \movethree. The case $\alpha = b+a$ is as with $\sbd C_a^b$, resulting in the block structure of $P$ ending in $U_a^b\wbd W_1^{b+a}$. Because of the definition of $U_b^a$, we know that $a\geq 3$, and thus we must consider $\alpha$ satisfying $b+2\leq \alpha\leq a+b-1$. Here we again use Lemma~\ref{lem:min-repn-semi} by noting that the interval $[b+1,\alpha]$ is contained in the interior of $[b,\alpha+1]$.

  The final case involving a strong boundary before the last block of
  $Q$ is when $Q$'s block structure ends with $\sbd W_a^b$. In this
  case, $\lQ[*]{Q} = b+a-1$. When $\alpha = b+a$, the $\Wb$ at the end
  grows to become $\sbd W_{a+1}^b$. If $\alpha = b+a-1$, we create
  duplicated holdings. If $a\geq 3$, having $\alpha \leq b+a-3$ is a
  \moveone\ which
  results in the interval $[b+a-2,b+a-1]$ being contained in the
  interior of $[\alpha,b+a]$, which takes us out of the class of
  semiorders. For $a\in\set{1,2}$, $\alpha < b$ places the interval
  $[b,b]$ in the interior of $[\alpha,b+a]$, so $P$ would not be a
  semiorder. Thus, it remains only to consider $a\geq 2$ and
  $\alpha = b+a-2$. Here, we have a \moveone{} that adds the interval
  $[b+a-2,b+a]$. When $a=2$, this converts the $\sbd W_2^b$ at the end
  of the block structure of $Q$ into $\sbd C_2^b$ at the end of the
  block structure of $P$. For $a > 2$, the block structure of $P$ ends
  $\sbd W_{a-2}^b\wbd C_2^{b+a-2}$. This is illustrated in
  Figure~\ref{fig:convert-W-W-C}.

  \begin{figure}[h]
    \centering
        \begin{tikzpicture}
      \draw[thick] (0,0.35) -- (0,0.6);
      \draw[thick] (4,-0.125) -- (4,0.125);
      
       \draw[{[-]}, thick] (0,0) -- (1,0);
       \draw[{[-]}, thick] (2,0) -- (3,0);
       \draw[{[-]}, thick] (1,0.475) -- (2,0.475);
       \draw[{[-]}, thick] (3,0.475) -- (4,0.475);

      \draw[->] (-0.2,-0.25) -- (4.2,-0.25);
      \foreach \x in {0,1,2,3,4}
      \draw (\x,-0.3) -- (\x,-0.2) node [below] {\scriptsize $b+\x$};
      \node [below] at (2,-.65) {$\sbd W_4^b$};
    \end{tikzpicture}\hspace*{0.15\linewidth}
    \begin{tikzpicture}
      \draw[thick] (0,0.35) -- (0,0.6);
      \draw[thick] (4,-0.125) -- (4,0.125);
      
       \draw[{[-]}, thick] (0,0) -- (1,0);
       \draw[{[-]}, thick] (2,0) -- (3,0);
       \draw[{[-]}, thick] (1,0.475) -- (2,0.475);
       \draw[{[-]}, thick] (3,0.475) -- (4,0.475);
       \draw[{[-]}, thick] (2,0.95) -- (4,0.95);

      \draw[dotted] (2,-1) -- (2,1.4);
      \draw[->] (-0.2,-0.25) -- (4.2,-0.25);
      \foreach \x in {0,1,2,3,4}
      \draw (\x,-0.3) -- (\x,-0.2) node [below] {\scriptsize $b+\x$};
      \node [below] at (1,-.65) {$\sbd W_2^b$};
      \node [below] at (3,-.65) {$\wbd C_2^{b+2}$};
    \end{tikzpicture} 
    \caption{Converting $\protect\sbd W_4^b$ to
      $\protect\sbd W_2^b\protect\wbd C_2^{b+2}$}
    \label{fig:convert-W-W-C}
  \end{figure}
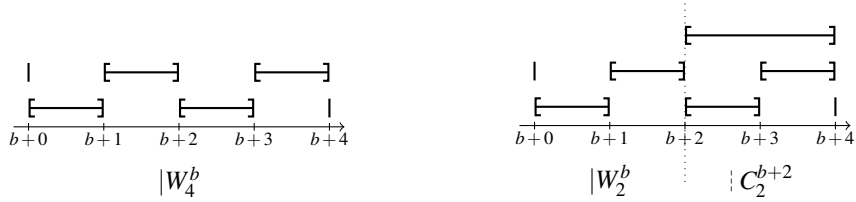
  
  To begin consideration of where the last boundary is weak, we assume
  that $Q$'s block structure ends $\wbd C_a^b$. Here $\lQ[*]{Q} = b$. The \moveone{} cases are $\alpha \leq b$. When $\alpha = b$, we have duplicated holdings. If $\alpha < b$, then since $a\geq 2$ the interval $[b,b+a-1]$ exists and is contained in the interior of $[\alpha,b+a]$; therefore, $P$ is not a semiorder. For $\alpha = b+1$, the situation is just as with a strong boundary, and the last block of $P$'s block structure is $\wbd U_{a+1}^b$. When $\alpha$ satisfies $b+2\leq \alpha\leq a+b$, the two possibilities are just as with $\sbd C_a^b$.

  When $Q$'s block structure ends $\wbd U_a^b$, the argument is
  identical to the $\sbd U_a^b$ case. Thus, we proceed to assume that
  the last boundary and block of $Q$ is $\wbd W_a^b$, which means that
  $\lQ[*]{Q} = b+a-1$. For $a\geq 2$ or $\alpha\geq b$, the situation is just as when the final boundary is strong. Thus, we must only consider when $a=1$ and $\alpha < b$. When $\alpha \leq b-2$, we note that the existence of a weak boundary means that the block preceding our $W_1^b$ must be a $\Cb$ or a $\Ub$, and thus the minimal endpoint representation of $Q$ contains the interval $[b-1,b]$. This interval is contained in the interior of the new interval $[\alpha,b+1]$, and $P$ is not a semiorder. The final case is $\alpha = b-1$, which leads us to the necessity of the optional interval, since the new interval is $[b-1,b+1]$. Thus, the final block and boundary of $P$'s block structure is $\wbdo W_1^b$.

  The previous case has forced us now to consider the situation where
  the final boundary and block of $Q$ is $\wbdo W_1^b$, in which case
  $\lQ[*]{Q} = b-1$. Here $a=1$, so $\alpha = b+a+1= b+2$ is the same as all other cases. For $\alpha = b+1$, we extend the last $\Wb$ to become $\wbdo W_2^b$. The case $\alpha\leq b-2$ creates a nonsemiorder by Lemma~\ref{lem:min-repn-semi} as in the case of a weak boundary. For $\alpha = b-1$, we have duplicated holdings. It remains only to consider the case $\alpha = b$. Note that the weak boundary preceding the last block of $Q$ must be preceded by a $\Cb$ or a $\Ub$. In either case, the interval $[b-2,b]$ must be present because of the minimum size of such blocks. In $P$, this interval becomes $[b-2,b+1]$, and the optional interval is truncated to $[b-1,b]$. We now have one interval in the interior of another, and so $P$ is not a semiorder by Lemma~\ref{lem:min-repn-semi}.

  We now must consider the case where $Q$'s final boundary
  and block are $\wbdo W_2^b$, which gives $\lQ[*]{Q} = b+1$. Here we have that $\alpha = b + a + 1$ when $\alpha = b+3$. For $\alpha = b+2$, we extend the final $\Wb$ to $\wbdo W_3^b$. When $\alpha = b+1$, we have duplicated holdings. If $\alpha =b$, the $\Wb$ at the end of $Q$'s block structure becomes $\wbdo C_2^b$ in $P$. For $\alpha\leq b-1$, the new interval is $[\alpha,b+2]$, which contains in its interior the interval $[b,b+1]$. Therefore, $P$ is not a semiorder.

  When $a\geq 3$ and $Q$'s block structure ends $\wbdo W_a^b$, the
  argument proceeds as it did with $\wbd W_a^b$. Thus, the only case
  we must still address is when the final block and boundary of $Q$'s
  block structure is $\wbdo C_2^b$. Here $\lQ[*]{Q} = b$. The case
  $\alpha = b+3$ is taken care of because $\alpha = b+a+1$ here. When $\alpha = b+2$, this
  is the same as the $\alpha=b+a$ case for $\wbd W_a^b$,
  and we have that the block structure for $P$ ends with $\wbdo C_2^b\wbd
  W_1^{b+2}$. This is illustrated in Figure~\ref{fig:wbdo-C}. For $\alpha = b+1$, note that the optional interval
  $[b-1,b+1]$ extends to become $[b-1,b+2]$, which contains in its
  interior the interval $[b,b+1]$ that results from the \movethree{}
  truncating $[b,b+2]$. Thus, $P$ is not a semiorder by
  Lemma~\ref{lem:min-repn-semi}. When $\alpha =b$, we have duplicated
  holdings. When $\alpha \leq b-1$, the interval $[\alpha,b+2]$
  contains the interval $[b,b+1]$ in its interior, violating
  Lemma~\ref{lem:min-repn-semi}. Since this case did not require us to
  permit a weak boundary with optional element before any other types
  of blocks, our proof of the existence of the block structure is complete.

  \begin{figure}[h]
    \centering
        \begin{tikzpicture}
      \draw[thick] (2,-0.125) -- (2,0.125);
      \draw[{[-]}, thick] (-1,0.475) -- (0,0.475);
       \draw[{[-]}, thick] (0,0) -- (1,0);
       \draw[{[-]}, thick] (1,0.475) -- (2,0.475);
       \draw[{[-]}, thick] (0,0.95) -- (2,0.95);

       \draw[{[-]}, thick] (-1,1.425) -- (1,1.425);
       \node [above] at (0,1.425) {$*$};

       \draw[dotted] (0,-1) -- (0,2.1);
      \draw[->] (-1.2,-0.25) -- (2.2,-0.25);
      \foreach \x/\xlab in {-1/$b-1$,0/$b$,1/$b+1$,2/$b+2$}
      \draw (\x,-0.3) -- (\x,-0.2) node [below] {\scriptsize \xlab};
      \node [below] at (1,-.65) {$\wbdo C_2^b$};
    \end{tikzpicture}\hspace*{0.15\linewidth}
    \begin{tikzpicture}
      \draw[thick] (3,0.35) -- (3,0.6);
      \draw[{[-]}, thick] (-1,0.475) -- (0,0.475);
       \draw[{[-]}, thick] (0,0) -- (1,0);
       \draw[{[-]}, thick] (1,0.475) -- (2,0.475);
       \draw[{[-]}, thick] (0,0.95) -- (2,0.95);

       \draw[{[-]}, thick] (2,0) -- (3,0);
       \draw[{[-]}, thick] (-1,1.425) -- (1,1.425);
       \node [above] at (0,1.425) {$*$};

       \draw[dotted] (0,-1) -- (0,2.1);

       \draw[dotted] (2,-1) -- (2,2.1);
      \draw[->] (-1.2,-0.25) -- (3.2,-0.25);
      \foreach \x/\xlab in {-1/$b-1$,0/$b$,1/$b+1$,2/$b+2$,3/$b+3$}
      \draw (\x,-0.3) -- (\x,-0.2) node [below] {\scriptsize \xlab};
      \node [below] at (1,-.65) {$\wbdo C_2^b$};
      \node [below] at (2.5,-.65) {$\wbd W_1^{b+2}$};
    \end{tikzpicture} 
    
    \caption{Moving on from $\protect\wbdo C_2^b$}
    \label{fig:wbdo-C}
  \end{figure}
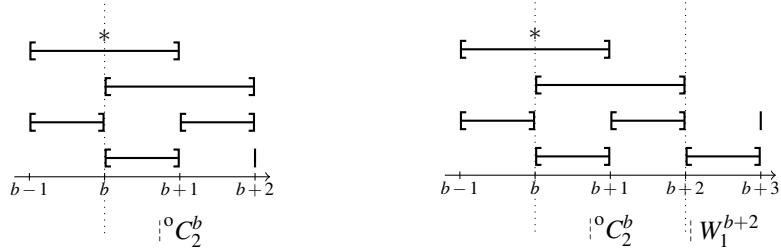

  It remains to show that the block structure of a hereditary
  semiorder $P$ is unique. To do so, we will identify the location and
  type of each boundary between blocks. Once this is done, the blocks
  between the boundaries are uniquely defined. To identify the
  boundaries, we begin by labeling all of the integers between $0$ and
  $\lQ{P}$ as follows:
  \[
    t(i) = \begin{cases}
      s & \text{if }[i,i]\text{ is in the representation}\\
      w & \text{if $i$ is the endpoint of at least 3 intervals and not
        in the interior of an interval}\\
      z &\text{if $i$ is the endpoint of at least 3 intervals and
        in the interior of an interval}\\
      x & \text{otherwise}.
    \end{cases}
  \]
  We have that $t(i) = s$ if and only if $i$ is the location of a
  strong boundary, since the block and boundary definitions only allow
  intervals of length $0$ at strong boundaries (including the implicit
  strong boundaries at the ends). Next, note that $t(i) =w$ if and
  only if $i$ is the location of a weak boundary. (This holds because
  we do not allow weak boundaries between $\Wb$.) We observe that the
  definition of $t$ tells us that if $t(i) = x$, then $i$ is not a
  boundary. To finish our argument, we will redefine $t(i)$ for those
  integers $i$ with $t(i) = z$. We wish to have $t(i) = o$ if and only
  if $i$ is the location of a weak boundary with optional interval and
  will define the other integers $j$ for which $t(j) = z$ to have
  $t(j) = x$. Let $i$ be the smallest integer in a maximal sequence of
  consecutive integers with label $z$. Since a weak boundary with
  optional interval must be preceded by a $\Ub$ or a $\Cb$, we know
  that $i$ cannot be the location of a weak boundary with optional
  interval. This is because the left endpoint of an optional interval
  must be the left endpoint of two intervals of the preceding $\Ub$ or
  $\Cb$ as well as being in the interior of at least one interval of
  the preceding block. Thus, we let $t(i) = x$, which then tells us
  that there is a weak boundary with optional element at $i+1$, so we
  let $t(i+1) = o$. If $t(i+2) = z$, then we must change $t(i+2)$ to
  $x$, since we cannot have two weak boundaries with optional
  intervals at consecutive integers. This process continues until no
  integers in $[0,\lQ{P}]$ have label $z$, which means the boundaries
  have all been uniquely determined because all decisions are forced.
\end{proof}

A careful reading of the preceding proof will show why the definition
of a weak boundary with optional element is so restrictive. In
particular, the optional interval is only introduced when absolutely
necessary, and then the argument proceeds to consider what can develop
following an optional interval. The fact that an optional interval can
only be preceded by a $\Cb$ or a $\Ub$ comes from the fact that our
first optional interval arises in the $\wbd W_1^b$ case, which
requires a $\Cb$ or $\Ub$ before it because of the prohibition against
weak boundaries between $\Wb$s. The only other weak boundaries with
optional elements arise as a consequence of building up from the
$\wbdo W_1^b$ case, and thus cannot be preceded by a $\Wb$ either.

\section{Block Characterization of Dimension 2 Semiorders}

We are now prepared to use the blocks and boundaries introduced above
to provide a characterization of the semiorders of dimension $2$,
which we will eventually use to enumerate them. We begin with a
straightforward lemma that links the moves used to construct an
interval order from an ascent sequence to subposet structure. This
will be useful in connecting to Rabinovitch's forbidden subposet
characterization of the dimension $2$ semiorders.

\begin{lem}\label{lem:move23-subposet}
  Let $P$ be a poset. If $Q$ is a subposet of $P$ and $P'$ is a poset
  obtained from $P$ by \moveone{} or \movetwo, then $Q$ is a subposet
  of $P'$.
\end{lem}

\begin{proof}
  Since neither \moveone{} nor \movetwo{} changes any of the
  existing comparabilities in $P$ to form $P'$, $P$ is a subposet of
  $P'$. Thus, $Q$ is a subposet of $P'$ as well.
\end{proof}

A full description of the block structure of semiorders of dimension
$2$ will be accomplished through a few steps. We begin by showing that
all semiorders of dimension $2$ are hereditary.

\begin{thm}\label{thm:dim2-hereditary}
  Let $P$ be a semiorder. If $P$ is not hereditary, then $\dim(P) = 3$.
\end{thm}

\begin{proof}
  Let $P$ be a semiorder on $n$ points and let
  $(x_1,\dots,x_n) = \Psi(P)$ be the ascent sequence corresponding to
  $P$. Without loss of generality, we may assume that $P$ has no
  duplicated holdings. Since $P$ is not hereditary, there is some
  largest positive integer $k < n$ such that $Q =
  \Psi\inv((x_1,\dots,x_k))$ is not a semiorder. Since we know that
  $Q'=\Psi\inv((x_1,\dots,x_{k+1}))$ is a semiorder, $Q'$ does not
  contain $\oneplusthree$. However, $Q$ must contain \oneplusthree,
  since $Q$ is an interval order that is not a semiorder. Therefore,
  by Lemma~\ref{lem:move23-subposet}, $Q'$ is not obtained from $Q$ by
  \moveone{} or \movetwo{}. We consider the minimal endpoint
  representation of $Q$. By Lemma~\ref{lem:min-repn-semi}, this
  representation has two intervals $[a,b]$ and $[c,d]$ with $[c,d]$
  contained in the interior of $[a,b]$. Since the \movethree{} that
  obtains $Q'$ from $Q$ destroys the \oneplusthree, we must have that
  $b=\lQ{Q}$ and $a < x_{k+1}$. If $x_{k+1}\leq c$, then the minimal
  endpoint representation of $Q'$ contains the interval $[c+1,d+1]$
  and the interval $[x_{k+1},b+1]$, which implies that $Q'$ is not a
  semiorder by Lemma~\ref{lem:min-repn-semi}. If $x_{k+1} > d$, then
  the minimal endpoint representation of $Q'$ contains the interval
  $[a,x_{k+1}]$, which contains $[c,d]$ in its interior. This would
  force $Q'$ to not be a semiorder. Thus, we must have that $x_{k+1}$
  is an integer with $c < x_{k+1} \leq d$, forcing $c\neq d$.

  Since $d < b$ and we are considering the minimal endpoint
  representation of $Q$, there exists an interval $[d,f]$ in the
  representation. Moreover, if $f < b$, then the minimal endpoint
  representation of $Q'$ contains the interval $[d+1,f+1]$, and this
  interval is contained in the interior of $[x_{k+1},b+1]$. This would
  prevent $Q'$ from being a semiorder, so $f=b$. Also note that there
  must be an interval $[g,c]$. If $g > a$, then the minimal endpoint
  representation of $Q'$ contains the interval $[g,c]$ and the
  interval $[a,x_{k+1}]$ with $x_{k+1} > c$. This again violates
  Lemma~\ref{lem:min-repn-semi}. The structural information we have
  gleaned so far is depicted in Figure~\ref{fig:dim2-step1}. Using
  what we know about $x_{k+1}$, we can draw
  Figure~\ref{fig:dim2-step2} to reflect intervals that must exist in
  $Q'$. It is straightforward to verify that these intervals give us
  the three-dimensional semiorder $\mbf{FX}_2$ from Figure~\ref{fig:semi-forb}.

  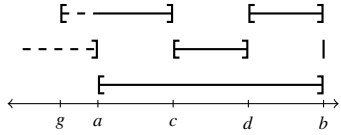
\begin{figure}[h]
    \centering
        \begin{tikzpicture}
      \draw[thick] (4,0.35) -- (4,0.6); 
       \draw[{[-]}, thick] (2,0.475) -- (3,0.475); 
      \draw[{-]}, thick, dashed] (0,0.475) -- (1,0.475); 

       \draw[{[-]}, thick] (1,0) -- (4,0); 

       \draw[{[-]}, thick] (3,0.95) -- (4,0.95); 
       \draw[{-]}, thick] (1,0.95) -- (2,0.95); 
       \draw[{[-}, thick, dashed] (0.5,0.95) -- (1,0.95); 

       \draw[<->] (-0.2,-0.25) -- (4.2,-0.25);
      \foreach \x/\xlab in {0.5/$g$,1/$a$,2/$c$,3/$d$,4/$b$} {
      \draw (\x,-0.3) -- (\x,-0.2);
      \node at (\x,-0.5) { \scriptsize\xlab};};
    \end{tikzpicture} 
    \caption{Intervals that must exist before eliminating
      \oneplusthree}
    \label{fig:dim2-step1}
  \end{figure}

  \begin{figure}[h]
    \centering
        \begin{tikzpicture}
      \draw[thick] (6,0.35) -- (6,0.6); 
       \draw[{[-]}, thick] (2,0.475) -- (4,0.475); 
      \draw[{-]}, thick, dashed] (0,0.475) -- (1,0.475); 

       \draw[{[-]}, thick] (1,0) -- (3,0); 
       \draw[{[-]}, thick] (4,0) -- (6,0); 

       \draw[{[-]},thick] (3,0.95) -- (6,0.95); 
       \draw[{-]}, thick] (1,0.95) -- (2,0.95); 
       \draw[{[-}, thick, dashed] (0.5,0.95) -- (1,0.95); 

       \draw[<->] (-0.2,-0.25) -- (6.2,-0.25);
      \foreach \x/\xlab in
      {0.5/$g$,1/$a$,2/$c$,4/$d+1$,3/$x_{k+1}$,5/$b$, 6/$b+1$} {
      \draw (\x,-0.3) -- (\x,-0.2);
      \node at (\x,-0.5) { \scriptsize\xlab};};
    \end{tikzpicture} 
    \caption{Intervals that must exist after eliminating
      \oneplusthree}
    \label{fig:dim2-step2}
  \end{figure}
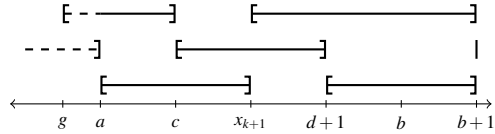

  By assumption, once we have obtained $Q'$ from
  $(x_1,\dots,x_{k+1})$, each of the posets obtained from $(x_1,\dots,
  x_m)$ with $m\geq k+1$ is a semiorder. We will show that it is
  impossible to eliminate all copies of $\mbf{FX}_2$ subject to this
  constraint, and thus we must have $\dim(P)=3$. To do so, assume that
  $m$ is such that $R = \Psi\inv((x_1,\dots,x_m))$ contains $\mbf{FX}_2$ and that for
  all $m > m'$, $\Psi\inv((x_1,\dots,x_{m'}))$ does not contain
  $\mbf{FX}_2$. We cannot be as precise about the endpoints as we were
  above at the first occurrence of $\mbf{FX}_2$, but we do have the
  configuration shown in Figure~\ref{fig:dim2-step3}. Note that we
  do not necessarily have that the endpoints shown as equal (such as
  $\iright{a_1}$ and $\ileft{b_2}$) are equal. Instead, we merely
  require that the intervals overlap.
    \begin{figure}[h]
    \centering
    \begin{tikzpicture}
      \node [above] at (5.5,0.475) {\scriptsize$c$};
      \draw[{[-}, thick] (5,0.475) -- (6,0.475); %
      \node [above] at (0.5,0.475) {\scriptsize$a_2$};
      \draw[{[-]}, thick] (2,0.475) -- (4,0.475); 
      \node [above] at (3,0.475) {\scriptsize$b_2$};
      \draw[{-]}, thick, dashed] (0,0.475) -- (1,0.475); 
      
      \node [above] at (2,0) {\scriptsize $a_3$};
      \draw[{-]}, thick] (1,0) -- (3,0); 
      \draw[thick, dashed] (0,0) -- (1,0); 
      \node [above] at (5,0) {\scriptsize $b_3$};
       \draw[{[-}, thick] (4,0) -- (6,0); 

       \node [above] at (4.5,0.95) {\scriptsize $b_1$}; 
       \draw[{[-},thick] (3,0.95) -- (6,0.95); 
       \node [above] at (1.2,0.95) {\scriptsize $a_1$};
       \draw[{-]}, thick] (1,0.95) -- (2,0.95); 
       \draw[{-}, thick, dashed] (0.5,0.95) -- (1,0.95); 

       \draw[<->] (-0.2,-0.25) -- (6.2,-0.25);
      \foreach \x/\xlab in
      {1/$a$,2/$c$,4/$d+1$,3/$x_{k+1}$,5/$b$, 6/$b+1$} {
      \draw (\x,-0.3) -- (\x,-0.2);
    };
    \end{tikzpicture} 
    \caption{Intervals forming an $\mbf{FX}_2$ in $R$}
    \label{fig:dim2-step3}
  \end{figure}
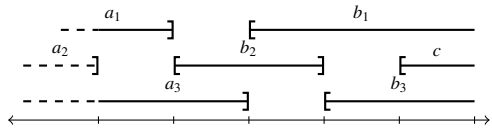

  By Lemma~\ref{lem:move23-subposet}, we only need to consider the effect of
  \movethree. If $x_{m+1}\leq \ileft{b_1}$, then the only impact of the
  \movethree{} on these seven intervals is stretching or shifting that
  does not impact their relationship to one another, and thus the
  $\mbf{FX}_2$ is not removed. Note that if none of $b_1,c,b_3$ is
  maximal in $R$, then this copy of $\mbf{FX}_2$ cannot be removed. If
  $b_1$ is maximal in $R$, then a \movethree{} with
  $\ileft{b_1}<x_{m+1}\leq \ileft{b_3}$ truncates $b_1$ and moves the
  right endpoint of $b_2$'s interval one unit right. This places $b_1$
  in the interior of $b_2$, violating the requirement that we must
  obtain a semiorder. (If $b_1$ is not maximal in $R$, a \movethree{}
  with $x_{m+1}$ in this range does not remove the $\mbf{FX}_2$.) A
  \movethree{} with $x_{m+1} > \iright{b_2}$ leaves a $\mbf{FX}_2$,
  either consisting of the same points (possibly with truncated
  intervals) or with the new interval playing the role of $c$ (and
  possibly with the intervals for $b_1$ and $b_3$ being truncated). A
  \movethree{} with $x_{m+1}$ satisfying $\ileft{b_3} <x_{m+1}\leq
  \iright{b_2}$ must truncate at least one of $b_1$ and $b_3$ if the
  $\mbf{FX}_2$ is to be eliminated. However, then the truncated
  interval lies in the interior of the stretched interval for $b_2$,
  and the resulting poset would not be a semiorder.

  Having shown that we cannot eliminate the last occurence of an
  $\mbf{FX}_2$ after the last occurrence of a \oneplusthree, we can
  therefore conclude that if $P$ is not hereditary, then $\dim(P)=3$
  as claimed.
\end{proof}

We now know that our search for semiorders of dimension at most $2$
can be restricted to the hereditary semiorders. Thus, we will proceed
to consider the three forbidden subposets of
Figure~\ref{fig:semi-forb} and what restrictions we must place upon
the block structure of a hereditary semiorder in order to exclude them.

\begin{lem}\label{lem:FX2}
  Let $P$ be a hereditary semiorder. If $P$ contains $\mbf{FX}_2$, then
  the block structure of $P$ requires an optional interval.
\end{lem}

\begin{proof}
  First note that $b_2$ is incomparable to $a_1$, $a_3$, $b_1$, and
  $b_3$, but $\set{a_1,a_3,b_1,b_2,b_3}$ is not a $5$-element
  antichain. Therefore, the interval corresponding to $b_2$ in the
  minimal endpoint representation of $P$ must have positive
  length. Since $a_3$, $b_1$, and $b_2$ are pairwise incomparable,
  their intervals must overlap. Let $x$ be an integer in the
  intersection of the intervals for $a_3$, $b_1$, and $b_2$. Since
  $a_3 < b_3$, we have that $x < \ileft{b_3}$.  Similarly,
  $\iright{a_1} < x$.  Thus $\iright{a_1}$, $x$, and $\ileft{b_3}$ are
  all distinct points in the interval for $b_2$. Hence, this interval
  has length at least $2$. If $b_2$ is an optional interval in the
  block structure, then we are done. If $b_2$ is not an optional
  interval, then since its length is at least $2$, it must lie in a
  $\Cb$ or a $\Ub$. Furthermore, at least one endpoint of $b_2$ must
  be the endpoint of the block containing $b_2$. Since $b_1$ is
  incomparable to $c$ and $b_2 < c$, we must have that the interval of
  $b_1$ extends to the right of $\iright{b_2}$. Since $a_3$ is
  incomparable to $a_2$ and $a_2 < b_2$, we must also have that the
  interval of $a_3$ extends to the left of $\ileft{b_2}$. Since $x$
  lies in both the interval of $a_3$ and that of $b_1$, this forces
  one of $b_1$ and $a_3$ to have its endpoints in two different
  blocks, and therefore, there must be an optional interval.
\end{proof}

\begin{lem}\label{lem:H0}
  Let $P$ be a hereditary semiorder. If $P$ contains $\mbf{H}_0$, then
  the block structure of $P$ requires an optional interval.
\end{lem}

\begin{proof}
  As before, we will assume that we are working with the minimal
  endpoint representation of $P$. Since $b_2$ is incomparable to
  $a_1$, $a_2$, $c$, and $d$, but $a_1 < a_2$, we must have that the
  length of $b_2$'s interval is at least $1$. Since $b_1 < b_2$ but
  $b_1$ is incomparable to $d$, we must have
  $\ileft{d} < \ileft{b_2}$. Similarly, since $b_2 < b_3$ and $b_3$ is
  incomparable to $c$, we must have $\iright{b_2} < \iright{c}$. Since $a_1 < a_2$ but
both $a_1$ and  $a_2$ are incomprable to $d$, we must have that the interval of $d$
  extends left of the interval of $a_2$. This gives $\ileft{a_2}\leq \iright{d}$. Further, the interval of
  $a_2$ must leave room for the interval of $a_1$ to intersect that of
  $b_2$, which requires $\ileft{b_2} < \ileft{a_2}$. Combining these
  inequalities gives $\ileft{b_2} < \ileft{a_2}\leq \iright{d}$. We
  may now conclude that $\ileft{b_2}$ lies in the interior of $d$. By
  the dual argument, we have that $\iright{b_2}$ lies in the interior
  of $c$. By the minimality of the representation, this forces
  the intervals of $c$ and $d$ to each have length at least $2$. If
  either of these is an optional interval, then we are done. If not,
  then they cannot belong to a $\Wb$ because of their intervals'
  lengths. Thus, the endpoints of $b_2$ lie in the interiors of two
  different blocks, which is only possible if $b_2$ is an optional
  interval.
\end{proof}

\begin{lem}\label{lem:G0}
  If $P$ is a hereditary semiorder containing a $\Cb$ somewhere other
  than the first or last block, then at least one of the blocks adjacent to the
  $\Cb$ is $T_1$ or $P$ contains $\mbf{G}_0$.
\end{lem}

\begin{proof}
  Suppose that $P$ is a hereditary semiorder containing $C_n^b$ with
  $b\neq 0$ and at least one following block. We also assume that neither
  neighboring block is $T_1$. Then the $C_n^b$ contains the intervals
  $[b,b+n]$, $[b,b+1]$, and $[b+1,b+n]$. The preceding block contains
  an interval containing the interval $[b-1,b]$ and an interval with
  right endpoint $b-1$. The succeeding block contains an interval
  containing $[b+n,b+n+1]$ and an interval with left endpoint
  $b+n+1$. These intervals, which are depicted in Figure~\ref{fig:C-forces-G0}, form a copy of
  $\mbf{G}_0$ in $P$.
  \begin{figure}[h]
    \centering
    \begin{tikzpicture}
      \draw[{[-]}, thick] (1,0) -- (2,0); 
      \node [above] at (1.5,0) {\scriptsize $a_2$};
    \draw[{[-}, thick] (4,0) -- (5,0); 
    \draw[thick, dashed] (5,0) -- (5.7,0); 
    \node [above] at (4.5,0) {\scriptsize $a_3$};
    
    \draw[{[-]}, thick] (1,0.475) -- (4,0.475); 
    \node [above] at (2.5,0.475) {\scriptsize $c$};
    \draw[{-]}, thick, dashed] (-0.7,0.475) -- (0,0.475); 
    \node [below] at (0,0.39) {\scriptsize $a_1$};
    
    \draw[{-]}, thick] (0,0.95) -- (1,0.95); 
    \draw[thick,dashed] (-0.7,0.95) -- (0,0.95); 
    \node [above] at (0.5,0.95) {\scriptsize $b_1$}; 
    
    \draw[{[-]}, thick] (2,0.95) -- (4,0.95); 
    \node [above] at (3,0.95) {\scriptsize $b_2$};
    \draw[{[-}, thick, dashed] (5,0.95) -- (5.7,0.95); 
    \node [above] at (5,1) {\scriptsize $b_3$}; 

    \draw[<->] (-1,-0.25) -- (6.2,-0.25);
    \foreach \x/\xlab in {0/$b-1$,1/$b$,2/$b+1$,3/$\cdots$,4/$b+n$,5/$b+n+1$}
       \draw (\x,-0.3) -- (\x,-0.2) node [below] {\scriptsize\xlab};
  \end{tikzpicture}

    \caption{A $\Cb$ with neighbors other than $T_1$ forcing
      $\mbf{G}_0$}
    \label{fig:C-forces-G0}
  \end{figure}
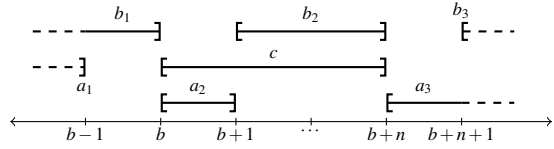
\end{proof}

\begin{lem}\label{lem:opt-dim3}
  If $P$ is a hereditary semiorder and the block structure of $P$
  requires an optional interval, then $\dim(P)=3$.
\end{lem}

\begin{proof}
  The proof is by straightforward case analysis based on what the
  blocks on either side of a weak boundary with optional interval can
  be. Recall that a weak boundary with optional interval must be
  preceded by a $\Cb$ or a $\Ub$ and must be followed by a $\Wb$ or a
  $C_2^b$, which limits the cases required. The cases and which
  forbidden subposet is produced are listed below.
  \begin{enumerate}
  \item $C_n^{b}\wbdo C_2^{b+n}$ for $n\geq 2$ and $U_n^b\wbdo
    C_2^{b+n}$ for $n\geq 3$ both produce $\mbf{FX}_2$.
  \item $C_n^b\wbdo W_1^{b+n}$ for $n\geq 2$ and $U_n^b\wbdo
    W_1^{b+n}$ for $n\geq 3$ both produce $\mbf{H}_0$.
    
  \item $C_2^b\wbdo W_m^{b+2}$ with $m\geq 2$ produces $\mbf{H}_0$.
  \item $C_n^b\wbdo W_m^{b+n}$ and $U_n^b\wbdo W_m^{b+n}$ with $n\geq
    3$ and $m\geq 2$ both produce $\mbf{FX}_2$.
  \end{enumerate}

  The first case is illustrated in Figure~\ref{fig:wbdo-C_2}. Note
  that if the $C_2^b$ is followed by a weak boundary, there is some
  interval from the next block with its left endpoint at $b+n$ that
  can be used as $c$. A similar situation applies if the block before the
  weak boundary is $C_2^{b-2}$ preceded by a weak
  boundary. Figure~\ref{fig:wbdo-C_2} is drawn to be general enough to
  encompass a $\Cb$ or $\Ub$ as the preceding block, and note that
  some intervals not involved in the $\mbf{FX}_2$ are omitted.
  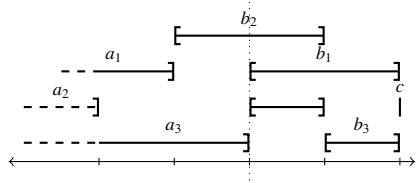
\begin{figure}[h]
    \centering
    \begin{tikzpicture}
      \draw[thick] (5,0.35) -- (5,0.6); 
      \node [above] at (5,0.55) {\scriptsize$c$};
      \node [above] at (0.5,0.475) {\scriptsize$a_2$};
      \draw[{-]}, thick, dashed] (0,0.475) -- (1,0.475); 
      \draw[{[-]},thick] (3,0.475)--(4,0.475);

      \draw[{[-]}, thick] (2,1.425) -- (4,1.425); 
      \node [above] at (3,1.425) {\scriptsize$b_2$};

      \node [above] at (2,0) {\scriptsize $a_3$};
      \draw[{-]}, thick] (1,0) -- (3,0); 
      \draw[thick, dashed] (0,0) -- (1,0); 
      \node [above] at (4.5,0) {\scriptsize $b_3$};
       \draw[{[-]}, thick] (4,0) -- (5,0); 

       \node [above] at (4,0.95) {\scriptsize $b_1$}; 
       \draw[{[-]},thick] (3,0.95) -- (5,0.95);
       \node [above] at (1.2,0.95) {\scriptsize $a_1$};
       \draw[{-]}, thick] (1,0.95) -- (2,0.95); 
       \draw[{-}, thick, dashed] (0.5,0.95) -- (1,0.95); 

       \draw[dotted] (3,-0.5) -- (3,2);
       
       \draw[<->] (-0.2,-0.25) -- (5.2,-0.25);
      \foreach \x/\xlab in
      {1/$a$,2/$c$,4/$d+1$,3/$x_{k+1}$,5/$b$} {
      \draw (\x,-0.3) -- (\x,-0.2);
    };
    \end{tikzpicture} 
     \caption{Intervals forming an $\mbf{FX}_2$ with $\protect\wbdo C_2^{b+n}$}
     \label{fig:wbdo-C_2}
   \end{figure} 

   The third case is illustrated in
   Figure~\ref{fig:C_2-wbdo-W}. Again, a weak boundary before the
   $C_2^b$ is not a problem, since there must be an interval from the
   previous block with its right endpoint at
   $b$. Figure~\ref{fig:C_2-wbdo-W} can be readily extended to the
   left in the style of Figure~\ref{fig:wbdo-C_2} to cover the second
   case as well, provided that one turns $b_3$ into an interval of
   length $0$ (or uses an interval from the next block if the
   following boundary is weak).
     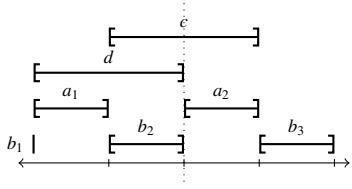
\begin{figure}[h]
    \centering
    \begin{tikzpicture}
      \draw[thick] (0,-0.125) -- (0,0.125); 
      \node [left] at (0,0) {\scriptsize$b_1$};
      \node [above] at (0.5,0.475) {\scriptsize$a_1$};
      \draw[{[-]},thick] (0,0.475)--(1,0.475);
      \node [above] at (3.5,0) {\scriptsize$b_3$};
      \draw[{[-]}, thick] (3,0)--(4,0);
      
      \node [above] at (2.5,0.475) {\scriptsize$a_2$};
      \draw[{[-]},thick] (2,0.475)--(3,0.475);

      \draw[{[-]}, thick] (1,1.425) -- (3,1.425); 
      \node [above] at (2,1.425) {\scriptsize$c$};

      \node [above] at (1.5,0) {\scriptsize $b_2$};
       \draw[{[-]}, thick] (1,0) -- (2,0); 

       \node [above] at (1,0.95) {\scriptsize $d$}; 
       \draw[{[-]},thick] (0,0.95) -- (2,0.95);

       \draw[dotted] (2,-0.5) -- (2,2);
       
       \draw[<->] (-0.2,-0.25) -- (4.2,-0.25);
      \foreach \x/\xlab in
      {1/$a$,2/$c$,4/$d+1$,3/$x_{k+1}$} {
      \draw (\x,-0.3) -- (\x,-0.2);
    };
    \end{tikzpicture} 
     \caption{Intervals forming a $\mbf{H}_0$ with
       $C_2^b\protect\wbdo W_m^{b+2}$ and $m\geq 2$}
     \label{fig:C_2-wbdo-W}
   \end{figure}

   The final case is not illustrated, but it is straightforward to
   verify after noting that the optional interval and the first two
   intervals of length $1$ from the $\Wb$ are $b_1$, $b_3$, and $c$ (in
   order by increasing left endpoint).
\end{proof}

We are now ready to assemble the preceding results to prove the block
characterization of semiorders of dimension at most $2$.

\begin{thm}\label{thm:dim2}
  Let $P$ be a semiorder. The dimension of $P$ is at most $2$ if and only
  if all of the following hold:
  \begin{enumerate}
  \item $P$ is hereditary,\label{item:hereditary}
  \item the block structure of $P$ does not require optional elements,
    and\label{item:no-opt}
  \item if the block structure of $P$ contains a $\Cb$ somewhere other
    than the first or last block, then at least one of the blocks adjacent to the
  $\Cb$ is $T_1$.\label{item:C-T1}
\end{enumerate}
\end{thm}

\begin{proof}
  We first assume that $P$ is a semiorder and $\dim(P)\leq 2$ and will
  prove that the three statements hold. The first statement is the
  contrapositive of Theorem~\ref{thm:dim2-hereditary}. The second
  statement is the contrapositive of Lemma~\ref{lem:opt-dim3}. Since
  $\dim(P)\leq 2$, $P$ does not contain $\mbf{G}_0$, and therefore the
  third statement follows from Lemma~\ref{lem:G0}.

  Now suppose that $P$ is a semiorder for which all three statements
  hold. We will show that $\dim(P)\leq 2$. The first two statements
  and the contrapositives of Lemmas~\ref{lem:FX2} and \ref{lem:H0}
  show that $P$ does not contain $\mbf{FX}_2$ or $\mbf{H}_0$. For a
  contradiction, we now assume that $P$ has dimension $3$. By what
  we've already shown, this means that $P$ must contain
  $\mbf{G}_0$. We will now show that the third statement must be
  violated by finding a $\Cb$ with two neighboring blocks that are not
  $T_1$. Since $c$ is incomparable to $a_2$, $a_3$, $b_1$, and $b_2$
  but $a_2 < a_3$, we know that the length of $c$'s interval in the
  minimal endpoint representation must be at least $1$. Since there
  are no duplicated holdings amongst the seven points of $\mbf{G}_0$,
  all intervals must be distinct. Furthermore, since an interval
  contained in the interior of a $\Wb$ is incomparable to only two
  other intervals but $c$ is incomparable to $4$ points from
  $\mbf{G}_0$, we know that one endpoint of $c$'s interval is on the
  boundary between two blocks. By duality, we may assume without loss
  of generality that this endpoint is $\ileft{c}$.

  Since $a_1 < c$ and
  $b_1$ is incomparable to both $c$ and $a_1$, we must have that $a_1$
  and $b_1$ both lie in a block before the one containing $c$ and that
  $\iright{b_1} = \ileft{c}$ because there are no optional intervals. Since $b_1$ has a larger up set than
  $a_2$ but is incomparable to $a_2$, we know that
  $\iright{b_1} < \iright{a_2}$. Thus, $a_2$ belongs to the same block
  as $c$ and $\ileft{c} = \ileft{a_2}$ because of the lack of optional
  intervals. Since $a_2$ has a larger up
  set than $c$, we know that $\iright{a_2} < \iright{c}$. Since $b_2$
  is incomparable to $a_2$, we can thus conclude that
  $\ileft{b_2} < \iright{c}$. If $\iright{b_2} > \iright{c}$, then we
  know by the lack of optional elements that $\iright{b_2}$ must be
  the right end of the block containing $c$, $a_2$, and $b_2$. Since
  $b_2 < b_3$, this means that $b_3$ must lie in a subsequent
  block. We now notice that $a_3$ is incomparable to $b_3$ and $c$,
  and thus $\iright{b_2}$ lies in the interior of $a_3$'s interval,
  which requires that $a_3$ cross the boundary of a block,
  contradicting the lack of optional intervals. Thus, we
  must have $\iright{b_2} = \iright{c}$ is the right boundary of the
  block containing $c$. Since $c$ has both its endpoints on the
  boundary of its block, we must have that $c$, $a_2$, and $b_2$ lie
  in a $\Cb$. Recognizing that $a_3$ is incomparable to both $c$ and
  $b_3$ and $b_1$ is incomparable to both $c$ and $a_1$ shows that
  neither of the adjacent blocks to the $\Cb$ can be $T_1$, and our
  proof is complete.
\end{proof}

\section{Enumeration of Hereditary Semiorders}

We are now prepared to use the block structure in order to enumerate
the hereditary semiorders. The difficulty in building a generating
function to complete this task is that there are restrictions on how
the blocks can be combined using the various boundaries. In
particular, we recall that a $T_1^b$ may not be combined using a weak
boundary, two consecutive $\Wb$ may not be combined using a weak
boundary, a $\Wb$ may not be followed by a weak boundary with optional
element, and a weak boundary with optional element may only preceed a
$\Wb$ or a $C_2^b$. We will use notation inspired by regular
expressions to give a compact way of describing the ways in which
blocks are arranged. The components of our notation are $+$, $*$, and
$?$, used as superscripts. A superscript $+$ will denote one or more
consecutive occurrences of the entity to which the $+$ is attached. A
$*$ means that zero or more consecutive occurrences of the entity are
allowed. A $?$ means that at most one occurrence of the entity is
allowed. When a $+$ is used between two strings (rather than a
superscript), each pattern is allowed. All of our boundaries
will be assumed to be weak unless explicitly shown in the
notation. Recall from Definition~\ref{defn:blocks} that we will use
$\Bb$ to refer to a block that could be either a $\Cb$ or a $\Ub$.

We will break up the block structure of a hereditary semiorder based
upon the occurrences of strong boundaries and the occurences of 
weak boundaries with optional intervals. Because a strong boundary is
determined based on the presence of an interval of length $0$ and such
an interval must be present at the left of the first block of a
hereditary semiorder and at the right of the last block, we will treat
the the two ends of a block structure as if there are strong
boundaries there. If we first consider the situation where no weak
boundaries with optional intervals are allowed, then it suffices to
break the full block structure up into the pieces between strong
boundaries. We represent this as $\Xss = \Wb^? (\Bb^+
\Wb)^*\Bb^*$. Essentially, between two strong boundaries, we can view
the blocks as divided further by the occurrences of the $\Wb$, which
may not be adjacent (since all boundaries inside this string are
weak). Between $\Wb$s, we must have at least one $\Bb$. Notice that
this structure allows for there to be \emph{no} blocks between two
strong boundaries, which is what creates a $T_1^b$. We may further
repeat the pattern $\Xss$ as many times as required, which then
introduces strong boundaries into the overall block structure.

When an optional interval is present, we may trace backward from that
weak boundary with optional interval until we reach either another
weak boundary with optional interval or a strong boundary (including
the beginning of the block structure). Thus, we will now describe two
further subpatterns, one to cover what can occur between a strong
boundary and the first ensuring weak boundary with optional interval
(denoted $\Xso$) and the other to cover what occurs between two weak
boundaries with optional intervals (denoted $\Xoo$). To construct
$\Xso$, note that the block before a weak boundary with optional
element must be a $\Bb$, and certainly many of them are permitted, so
$\Xso$ must end with $\Bb^+$. Other than needing to end with a $\Bb$,
this case looks much like $\Xss$, in that we see isolated $\Wb$ with
strings of $\Bb$ in between, and an initial $\Wb$ may or may not
occur. Thus, $\Xso = W^?(\Bb^+\Wb)^*\Bb^+$. When both ends of a string
of blocks joined by weak boundaries are weak boundaries with optional
intervals, the situation is more complicated. The weak boundary with
optional interval may be followed by a $\Wb$, in which case the
structure proceeds just as with $\Xso$, since the final block of the
pattern must be a $\Bb$ to allow for the trailing weak boundary with
optional interval. This means $\Xoo$ must allow $\Wb(\Bb^+
\Wb)^*\Bb^+$. We may also follow the weak boundary with optional
interval with $C_2^b$. Since this block is itself a $\Bb$, this could
be the end of the pattern, proceeding immediately to another weak
boundary with optional interval. If not, we then see the remainder
divided up by $\Wb$, ensuring that the last block before the weak
boundary with optional interval is a $\Bb$. This gives us $C_2^b
\Bb^*(\Wb\Bb^+)^*$, which combines with the case where the first block
after the weak boundary with optional interval is $\Wb$ to give us
\[\Xoo = \Wb(\Bb^+\Wb)^*\Bb^+ + C_2^b \Bb^*(\Wb\Bb^+)^*.\]

As we proceed through the block structure, we must eventually reach an
occurrence of a weak boundary with optional interval where the next
meaningful boundary is a strong boundary (possibly the one at the end
of the block structure). Thus, we need a pattern to
describe what happens in such a case, which we denote by
$\Xos$. Again, the weak boundary with optional interval may be
followed by a $\Wb$ or a $C_2^b$. The former case gives rise to
$\Wb(\Bb^+\Wb)^*\Bb^*$, much like in $\Xoo$, but here we end with
$\Bb^*$ because the next boundary is strong, and so we may end with a
$\Wb$. When beginning with a $C_2^b$, the situation is also analogous
to $\Xoo$, but we must allow a $\Wb$ at the end, which gives
$C_2^b\Bb^*(\Wb\Bb^+)^*\Wb^?$. Combining these yields
\[\Xos = \Wb(\Bb^+\Wb)^*\Bb^* + C_2^b\Bb^*(\Wb\Bb^+)^*\Wb^?.\]

We now have all the pieces necessary to create a pattern that
describes the block structure of all hereditary semiorders. We first
note that a weak boundary with optional interval may occur in the form
$\Xso\wbdo\Xos$, or we may place several copies of $\Xoo$ (with weak
boundaries with optional intervals on each side) in between the $\Xso$
and the $\Xos$. This means we will need to see $\Xso\wbdo
(\Xoo\wbdo)^*\Xos$ in the overall pattern. Since there may be multiple
strong boundaries before the first weak boundary with optional
interval, the overall pattern must begin $\Xss^*$. We need another
occurrence of $\Xss^*$ along with the pattern containing weak
boundaries with optional intervals in order to allow  weak
boundaries with optional intervals to be separated by a combination of
strong and weak boundaries. Thus, the pattern that accounts for all
hereditary semiorders is
\[\mc{H} = (\Xss\sbd)^*(\Xso\wbdo(\Xoo \wbdo)^*\Xos\sbd (\Xss\sbd)^*)^*.\]
Note that $\mc{H}$ allows for the empty pattern, which is how we will
account for $T_0$ when converting this pattern into a generating function.

Translation of the $+$, $*$, and $?$ used in our patterns into
generating functions is relatively straightforward. For readers
unfamiliar with the use of generating functions to enumerate strings
or sequences in this manner, a good introduction is provided by Wilf
in \cite{wilf:genfn}. If $\mc{F}$ is a pattern with generating
function $F(x)$, then $\mc{F}^*$ has generating function $1/(1-F(x))$,
$\mc{F}^+$ has generating function $F(x)/(1-F(x))$, and $\mc{F}^?$ has
generating function $(1+F(x))$. The other piece that will require
attention is the boundaries, but first we will proceed to determine
the generating functions for $\Wb$, $\Ub$, $\Cb$, $\Bb$, and $C_2^b$,
since those are the atomic pieces of the patterns here. (The patterns
developed so far do not involve $\Cb$ or $\Ub$ alone, but we will
require these when considering the case of dimension at most $2$ in
the next section.)

Because our patterns above are built on the assumption of weak
boundaries between blocks unless we specify a strong boundary or weak
boundary with optional interval, we will build our generating
functions for the blocks by assuming weak boundaries on each end. This
then has the effect of making each of our blocks appear to have two
fewer intervals in them than they would when occurring in
isolation. For example, the smallest $\Cb$ is $C_2^b$, which has $5$
intervals. However, the lowest order term in $C(x)$, the generating
function for $\Cb$, will be $x^3$. Throughout the following, we will
use $F(x)$ as the generating function for the block or pattern
$\mc{F}$.

Recall that an interval order has duplicated holdings if and only if
two points of the interval order have the same interval in its minimal
endpoint representation. Also, the only way to create duplicated
holdings in an ascent sequence is to have $x_i = x_{i+1}$. Thus, we
may proceed to think about the blocks on the basis of no duplicated
holdings and then form the generating function by allowing repetition
of terms in the ascent sequence to allow for duplicated holdings. For
conciseness as we do this, we will let $f(x) = x/(1-x)$ for the
remainder of the paper.

When $W_k^b$ is preceded and followed by a weak boundary, we do not
have length $0$ intervals to concern ourselves with. The one that
would be present at the left with a strong boundary is simply never
created, and the one that would be at the right with a strong boundary
is created by the ascent sequence from an earlier block and
subsequently moved into a later block. Thus, we are concerned with a
subsequence of length $k$ when we work without duplicated
holdings. The subsequence we must have is $b,b+1,b+2,b+3,\dots,b+k-1$,
since each move $b+i$ takes the interval of length $0$ and shifts it
to the right while adding the interval $[b+i,b+i+1]$. Since $k\geq 1$,
we know that the subsequence is not empty. Thus, the generating
function component we need here (before allowing for duplicated
holdings) is $f/(1-f)$, since there is only one way to do
things. Substituting $f(x)$ for $f$ takes care of duplicated holdings
for us, since we may repeat the integers from $b$ to $b+k-1$ provided
they remain in increasing order and each integer appears at least one
time. Therefore, $W(x) = f(x)/(1-f(x))$.

Note that a $\Ub$ must have at least six intervals (and thus at least
four must be accounted for in our generating function). Also, the
intervals appear in pairs. A $U_k^b$ following a weak boundary is
created by the subsequence $b,b+1,b,b+1,\dots,b,b+1$, where there are
$k-1$ appearances of $b,b+1$. This is because the $b$ creates an
interval with its left endpoint at $b$ and its right endpoint as far
right as possible, and then the subsequent $b+1$ produces a
\movethree{} that truncates the interval created at the previous step
to end at $b+1$ and stretches/shifts the other intervals of the
block. Since $k\geq 3$, before duplicated holdings here we have
$f^4/(1-f^2)$, with the $f^4$ accounting for the initial $b,b+1,b,b+1$
and the $1/(1-f^2)$ providing the subsequent pairs
$b,b+1$. Substituting $f(x)$ for $f$ takes care of duplicated holdings
and gives us $U(x) = (f(x))^4/(1-(f(x))^2)$.

The situation for $\Cb$ is a slight modification of what we did for
$\Ub$ above, since if the subsequence ended with a $b$ instead of a
$b+1$, we would have the interval that spans the length of the
$\Cb$. Thus, the subsequence correspondint to $C_k^b$ must start
$b,b+1,b$ and then have $k-2$ pairs $b+1,b$ following it. Since $k\geq
2$, we have $f^3/(1-f^2)$ before addressing duplicated holdings.
Therefore $C(x) = (f(x))^3/(1-(f(x))^2)$. The generating function of $C_2^b$, which is
required in $\Xoo$ and $\Xos$, is $(f(x))^3$. Since $\Bb$ merely stands for a $\Cb$ or a $\Ub$, $B(x) =
U(x) + C(x)$, which simplifies to $(f(x))^3/(1-f(x))$.

Assembling the generating function is now a matter of introducing
additional factors of $f(x)$ for each strong boundary and each weak
boundary with optional element, since we known exactly what number
must appear in the ascent sequence to produce the required interval,
but we may repeat it as many times as we like to account for
duplicated holdings. Therefore, we have the following:
\begin{align*}
  X_{ss}(x) =&
              f(x)(1+W(x))\frac{1}{1-\frac{B(x)}{1-B(x)}W(x)}\frac{1}{1-B(x)}\\
  X_{so}(x) =&f(x)(1+W(x))\frac{1}{1-\frac{B(x)}{1-B(x)}W(x)}\frac{B(x)}{1-B(x)}\\
  X_{oo}(x) =& f(x)W(x)
              \frac{1}{1-\frac{B(x)}{1-B(x)}W(x)}\frac{B(x)}{1-B(x)}
              + f(x)\cdot (f(x))^3\cdot \frac{1}{1-B(x)}\frac{1}{1-W(x)\frac{B(x)}{1-B(x)}}\\
  X_{os}(x) =& f(x)W(x) \frac{1}{1-\frac{B(x)}{1-B(x)}W(x)}
              \frac{1}{1-B(x)}\\ &+ f(x)\cdot (f(x))^3 \cdot \frac{1}{1-B(x)} \frac{1}{1-W(x)\frac{B(x)}{1-B(x)}} (1+W(x))\\
  H(x) =& f(x) \frac{1}{1-X_{ss}(x)} \frac{1}{X_{so(x)} \frac{1}{1-X_{oo}(x)} X_{os}(x)\frac{1}{1-X_{ss}(x)}}
\end{align*}

After fully substituting, we conclude that this section has proved the following
theorem:
\begin{thm}\label{thm:hereditary-gf}
  The generating function for the number of hereditary semiorders with
  $n$ points is
  \[H(x) = \frac{-x^5 + 9x^4 - 12x^3 + 6x^2 - x}{x^5 - 14x^4 + 29x^3 - 23x^2 + 8x-1}.\]
\end{thm}
A table of values and discussion of asymptotics will be deferred to
section~\ref{sec:conc}, after we have completed our enumeration of the
semiorders of dimension at most $2$.

\section{Enumeration of Dimension 2 Semiorders}

The previous section has completed much of the work required for the
enumeration of the semiorders of dimension at most $2$, since we have
the necessary components to address each of the block types. However,
the rules for combining the blocks in this case are different. On one
level, things get simpler, because we no longer are allowed to have
weak boundaries with optional elements. However, the third statement
of Theorem~\ref{thm:dim2} places significant restrictions on how a
$\Cb$ may appear in the block structure of a semiorder of dimension at
most $2$. We can use this to our advantage, however, since an interior $\Cb$
must have a $T_1^b$ as a neighbor on (at least) one side, which means
that interior $\Cb$ must appear adjacent to a strong boundary.

We proceed by considering what can happen between occurrences of
$T_1$. The first pattern we consider represents when there are no
strong boundaries between two appearances of $T_1$ (other than the
strong boundaries necessitated by the $T_1$s). We call this pattern
$\Ab_0$. Because the blocks on either side of $\Ab_0$ are $T_1$, we
are allowed the option of a $\Cb$ as the first block of $\Ab_0$ or as
the last block of $\Ab_0$, but we may not have a $\Cb$ anywhere else
inside $\Ab_0$. What appears between these two possible $\Cb$ must be
a mix of $\Wb$ and $\Ub$, all combined by weak boundaries. Thus, the
interior must take the form $\Ub^*(\Wb\Ub^+)^*\Wb^?$. It is tempting
to sandwich this pattern between two $\Cb^?$ and be done, but that
would give us two distinct ways of getting a $\Cb$ by itself, which we
cannot allow. Thus, our definition is
\[\Ab_0 = \Cb^?\nonempty{\Ub^*(\Wb\Ub^+)^*\Wb^?}\Cb^? + \Cb\Cb^?,\]
where the $\nonempty{\cdot}$ indicates that we do not allow the enclosed
portion of the pattern to be empty. (This is readily accomplished in
the generating function by subtracting $1$ from the factor that would
otherwise be present.)

Next, we consider what happens when there are strong boundaries that
occur between the $T_1$s. The argument is essentially the same as
before, giving rise to the pattern
\[\Ab_s = \nonempty{\Cb^? \Ub^*(\Wb\Ub^+)^*\Wb^?}\sbd
  \nonempty{\Ub^*(\Wb\Ub^+)^*\Wb^?\sbd}^*
  \nonempty{\Wb^?(\Ub^+\Wb)^*\Ub^*\Cb^?}. \]
If we define $\Ab = \Ab_0 + \Ab_s$, then $\Ab$ represents whatever can
occur between two non-adjacent $T_1$ in a semiorder of dimension at
most $2$. This tells us that the pattern $\mc{D}$ that
represents all semiorders of dimension at most $2$ is 
\[\mc{D} = (\Ab\sbd)^?(T_1^+\sbd \Ab)^* T_1^*.\]

The conversion to a generating function proceeds as in the previous
section, including the introduction of an initial factor of $f(x)$ to
account for the interval $[0,0]$. Since the pattern $\mc{D}$ can be
empty, this factor will account for $T_0$ (and duplicated
holdings). We do need the generating function to introduce for $T_1^+$
and $T_1^*$. Because the subsequence required is prescribed and does
not involve any repetitive structure, we conclude that the former is
$f(x)/(1-f(x))$, while the latter is $1+(f(x)/(1-f(x)))$. After fully
substituting and simplifying, we can therefore conclude the following
theorem.

\begin{thm}\label{thm:dim2-gf}
  The generating function for the number of semiorders of dimension at
  most $2$ with $n$ points is
  \[D(x) = \frac{-5x^8 + 41x^7 - 101x^6 + 129x^5 - 96x^4 + 42x^3 -
      10x^2 + x)}{7x^8 - 66x^7 + 197x^6 - 311x^5 + 294x^4 - 172x^3 +
      61x^2 - 12x + 1}.\]
\end{thm}
\section{Conclusion}\label{sec:conc}

\subsection*{Exact and Asymptotic Values} Recalling that the number of semiorders on $n$ points is the
$n$\textsuperscript{th} Catalan
number, we can use SageMath \cite{sagemath} and the generating
functions from Theorems~\ref{thm:hereditary-gf} and \ref{thm:dim2-gf} to calculate the number of
semiorders on $n$ points, the number of hereditary semiorders on $n$
points, the number of semiorders of dimension at most $2$ on $n$
points, and the number of semiorders of dimension $3$ on $n$
points. These values are shown in Table~\ref{tab:counts}, with the
second line of each column header giving the sequence number in the
Online Encyclopedia of Integer Sequences.
\begin{table}[bh]\centering
  
  \begin{tabular}{crrrr}
    $n$ & Semiorders & Hereditary & $\dim\leq 2$ & $\dim =3$\\
    & A000108 & A293499 & A293498 & A293501\\\Xhline{2\arrayrulewidth}
    1 &                    1 &                    1 &                    1 &                    0\\
    2 &                    2 &                    2 &                    2 &                    0\\
    3 &                    5 &                    5 &                    5 &                    0\\
    4 &                   14 &                   14 &                   14 &                    0\\
    5 &                   42 &                   42 &                   42 &                    0\\\hline
    6 &                  132 &                  132 &                  132 &                    0\\
    7 &                  429 &                  428 &                  426 &                    3\\
    8 &                1,430 &                1,415 &                1,390 &                   40\\
    9 &                4,862 &                4,730 &                4,544 &                  318\\
    10 &               16,796 &               15,901 &               14,822 &                1,974\\\hline
    11 &               58,786 &               53,593 &               48,183 &               10,603\\
    12 &              208,012 &              180,809 &              156,118 &               51,894\\
    13 &              742,900 &              610,157 &              504,487 &              238,413\\
    14 &            2,674,440 &            2,058,962 &            1,627,000 &            1,047,440\\
    15 &            9,694,845 &            6,947,145 &            5,240,019 &            4,454,826\\\hline
    16 &           35,357,670 &           23,437,854 &           16,861,453 &           18,496,217\\
    17 &          129,644,790 &           79,067,006 &           54,228,190 &           75,416,600\\
    18 &          477,638,700 &          266,717,300 &          174,351,450 &          303,287,250\\
    19 &        1,767,263,190 &          899,693,960 &          560,481,708 &        1,206,781,482\\
    20 &        6,564,120,420 &        3,034,814,143 &        1,801,653,769 &        4,762,466,651\\\hline
    21 &       24,466,267,020 &       10,236,853,534 &        5,791,301,311 &       18,674,965,709\\
    22 &       91,482,563,640 &       34,530,252,629 &       18,615,976,402 &       72,866,587,238\\
    23 &      343,059,613,650 &      116,475,001,757 &       59,841,686,254 &      283,217,927,396\\
    24 &    1,289,904,147,324 &      392,885,252,033 &      192,366,897,839 &    1,097,537,249,485\\
    25 &    4,861,946,401,452 &    1,325,253,166,761 &      618,392,292,337 &    4,243,554,109,115\\\hline
  \end{tabular}
  
  \caption{Exact counts of the various classes of semiorders}
  \label{tab:counts}
\end{table}

An asymptotic analysis of the coefficients of the rational generating
functions derived above is a straightforward application of the
techniques of section IV.5 of \cite{flajolet:analytic-comb} by
Flajolet and Sedgewick. The poles of $H(x)$ are $1$ and approximately
$0.29646$, $11.681$, and $0.51131\pm0.16533i$. Thus, the
number of hereditary semiorders on $n$ points is asymptotically
$0.08346\cdot 3.373133^n$. The poles of $D(x)$ are approximately
$0.311065$, $5.60822$, $0.456557\pm 0.123792 i$,
$0.536649\pm 0.24759i$, and $0.761438\pm 0.68404i$. Thus, 
the number of semiorders of dimension at most $2$ on $n$ points is
asymptotically $0.12958\cdot 3.2148^n$. For comparison, recall that
the Catalan numbers are asymptotically $4^n/(n^{3/2}\sqrt{\pi})$.

\subsection*{No Duplicated Holdings}
As discussed in the arguments that led to Theorems~\ref{thm:hereditary-gf} and
\ref{thm:dim2-gf}, we use $f(x) = x/(1-x)$ in the construction of the
generating functions to allow for consecutive appearances of an
integer in the ascent sequences and therefore duplicated holdings in
the poset. If, instead, we write those generating functions in terms
of the variable $f$ (replacing any explicit occurrence of $f(x)$ by
$f$), we then have the following corollary.

\begin{cor}\label{cor:NODH-gf}
  The ordinary generating functions for the number of hereditary semiorders
  with no duplicated holdings ($H_N(f)$) and the number of semiorders of
  dimension at most $2$ with no duplicated holdings ($D_N(f)$) are
  \begin{align*}
    H_N(f) &= \frac{f^5 - f^4 + 2f^2 - f}{2f^4 - 2f^3 - f^2 + 3f - 1}\\
    \intertext{and}
    D_N(f) &=\frac{-f^8 + f^7 - f^6 + f^4 - 3f^3 + 3f^2 - f}{f^8 - f^7 + f^6 - f^5 + f^4 + 2f^3 - 5f^2 + 4f - 1}.
  \end{align*}
\end{cor}

\subsection*{Restricted Ascent Sequences} As mentioned in the
introduction, Kitaev and Remmel showed in \cite{kitaev:intord-stats}
that the Catalan numbers enumerate a nicely-defined subset of ascent
sequences. They called an ascent sequence $(x_1,\dots,x_n)$ a
\emph{restricted ascent sequences} if $x_1 = 0$ and for all $i$ with
$2\leq i\leq n$, $m-1\leq x_i\leq 1+\asc((x_1,\dots,x_{i-1}))$,
where 
$m$ is the largest term in $(x_1,\dots,x_{i-1})$. However, they also
showed that the restricted ascent sequences do not correspond to the
semiorders under the bijection $\Psi$. The ascent sequence
$(0,1,0,1,2,0,2)$ of Figure~\ref{fig:nonhered2} corresponds to a
semiorder, but the ascent sequence is not restricted. The sequence
$(0,1,0,1,0,1,2)$ is a restricted ascent sequence, but it is easy to
verify that it does not correspond to a semiorder. While we are not
able at this time to fully characterize the interval orders
corresponding to restricted ascent sequences, we do have the following
theorem as fairly direct consequence of our earlier work.

\begin{thm}\label{thm:restricted-hereditary}
  Let $P$ be a semiorder and $(x_1,\dots,x_n) = \Psi(P)$ the
  corresponding ascent sequence. The sequence $(x_1,\dots,x_n)$ is a
  restricted ascent sequence if and only if $P$ is hereditary.
\end{thm}

\begin{proof}
  When $P$ is hereditary, the fact that $(x_1,\dots,x_n)$ is a
  restricted ascent sequence follows primarily from the proof of the
  block structure in Theorem~\ref{thm:block-structure} and the proof
  of the enumeration of hereditary semiorders in
  Theorem~\ref{thm:hereditary-gf}. We have given the values of
  $\lQ[*]{Q_i}$, where $Q_i = \Psi\inv((x_1,\dots,x_{i}))$ in the
  proof of Theorem~\ref{thm:block-structure}. Using the proof of
  Theorem~\ref{thm:hereditary-gf}, It is straightforward to verify
  that if the last block is $T_1^b$, $C_a^b$, or $U_a^b$, then the
  maximum value $m$ in the ascent sequence is $b+1$. If the last block
  is $W_a^b$ and $a> 1$, then $m=b+a-1$. If the last block is $W_1^b$,
  then $m = b+1$ if the preceding boundary is strong and $m=b$ if the
  preceding boundary is weak. Since the only time we can add an
  interval that extends to the left of $b$ in a hereditary semiorder
  in any of these cases is when the last block is $W_1^b$ and we are
  adding the optional interval, we thus can see by induction that
  $(x_1,\dots,x_n)$ is a restricted ascent sequence.

  For the converse, we consider a minimal counterexample. That is, we
  assume that $(x_1,\dots,x_n)$ is a restricted ascent sequence
  corresponding to a semiorder but that it is not hereditary. Hence
  there is an integer $k<n$ such that for all $i\leq k$,
  $Q_i = \Psi\inv((x_1,\dots,x_i))$ is a semiorder but
  $Q_{k+1} = \Psi\inv((x_1,\dots,x_{k+1}))$ is not a semiorder. By
  what we have assumed, we know that $Q_k$ is a hereditary
  semiorder. Therefore, Theorem~\ref{thm:block-structure} describes
  its block structure. If the last block is $T_1^b$, $C_a^b$, or
  $U_a^b$, then $\max_{i\colon 1\leq i\leq k} x_i = b+1$. Thus,
  $x_{k+1}\geq b$, since we are working with a restricted ascent
  sequence. If $x_{k+1} = b$, then \moveone{} is used, which cannot
  create a \oneplusthree{} here. From the proof of
  Theorem~\ref{thm:block-structure}, we also know that if
  $x_{k+1} \in \set{b+1,b+a,b+a+1}$, then $Q_{k+1}$ is a
  semiorder. This leaves us to consider $a\geq 3$,
  $b+2\leq x_{k+1} < b+a$, and the last block $U_a^b$ or
  $C_a^b$. Here, the \movethree{} leaves us with $[b+1,b+2]$ in the
  interior of $[b,b+a]$, which results from stretching
  $[b,b+a-1]$. Since neither of these intervals reaches to the largest
  right endpoint of the minimal endpoint representation, this
  containment relationship cannot be changed, and
  $\Psi\inv((x_1,\dots,x_n))$ cannot correspond to a semiorder. When
  the last block is $W_a^b$ with $a > 1$, then the fact that
  $m = b+a-1$ prevents us from adding an interval that contains
  another in its interior. When the block structure of $Q_k$ ends
  $\sbd W_1^b$, we have that $m=b+1$, and thus we cannot add an
  interval creating a \oneplusthree. When the block structure of $Q_k$
  ends with $W_1^b$ preceded by a weak boundary (with or without
  optional interval), $m=b$. The largest interval we can thus add,
  given we have a restricted ascent sequence, is $[b-1,b+1]$, which
  does not create a \oneplusthree{} because the interval $[b,b]$ is
  not present. Therefore, a counterexample cannot exist, and our proof
  is complete.
\end{proof}

\subsection*{Open Questions} We close with some possible interesting
directions for future work. One would be to consider other classes of
combinatorial objects equinumerous to interval orders and ascent
sequences to see if there is another natural way to construct a
bijection between interval orders and ascent sequences in such a way
that every initial subsequence of an ascent sequence corresponding to
a semiorder is one that also corresponds to a semiorder. Put another
way, can we find a bijection $\Phi$ from interval orders to ascent
sequences so that replacing $\Psi$ by $\Phi$ in
Definition~\ref{defn:hereditary} leads to all semiorders being
hereditary?

Theorem~\ref{thm:restricted-hereditary} shows that the restricted
ascent sequences defined by Kitaev and Remmel in
\cite{kitaev:intord-stats} that correspond to semiorders give rise to
precisely the hereditary semiorders. We have left open the question of
characterizing all interval orders that correspond to restricted
ascent sequences.

Another direction of interest would be to discover more enumerative
results involving more global poset statistics. Most of the recent
restricted enumeration results focus on statistics that do not appear
frequently in the poset literature. (The exceptions being the work of
Khamis in \cite{khamis:intord-height} and Hu in
\cite{hu:semiorders-height}, where height was the driving statistic.)
Given an interval representation, the width of an interval order is
easy to calculate. An enumeration of interval orders (or semiorders)
by width would be of interest. Dimension would be another natural parameter
to attempt enumeration by, but since the dimension of interval orders
is unbounded, the problem is likely very hard.

\subsection*{Acknowledgments}

The authors would like to thank Jeffrey Remmel for his encouragement
and helpful conversations during the early stages of this work. We are
also grateful to Emilie Purvine and the two anonymous referees for
their thoughtful reading and comments on this paper.
\bibliographystyle{acm}
\bibliography{zotero-abbrev}

\end{document}